\documentclass[11pt,twoside]{article}

\usepackage{fullpage}
\usepackage{comment}
\usepackage{epsf}
\usepackage{graphicx}
\usepackage{bbm}
\usepackage{color}
\usepackage{nicefrac}
\usepackage{mathtools}
\usepackage{amsfonts}
\usepackage{amsthm}
\usepackage{amsmath, bm}
\usepackage{amssymb}
\usepackage{textcomp}
\usepackage{xcolor}
\usepackage[ruled, vlined, lined, commentsnumbered]{algorithm2e}
\usepackage{caption}
\usepackage{subcaption}

\RequirePackage[colorlinks,citecolor={blue!60!black},urlcolor={blue!70!black},linkcolor={red!60!black},breaklinks,hypertexnames=false]{hyperref} 

\usepackage{natbib}
\setcitestyle{round}

\newtheorem{theorem}{Theorem}

\newtheorem{assumption}{Assumption}
\newtheorem{proposition}{Proposition}
\newtheorem{lemma}{Lemma}

\newtheorem{definition}{Definition}

\newenvironment{fminipage}%
  {\begin{Sbox}\begin{minipage}}%
  {\end{minipage}\end{Sbox}\fbox{\TheSbox}}

\newlength{\widebarargwidth}
\newlength{\widebarargheight}
\newlength{\widebarargdepth}

\makeatletter
\long\def\@makecaption#1#2{
       \vskip 0.8ex
       \setbox\@tempboxa\hbox{\small {\bf #1:} #2}
       \parindent 1.5em
       \dimen0=\hsize
       \advance\dimen0 by -3em
       \ifdim \wd\@tempboxa >\dimen0
               \hbox to \hsize{
                       \parindent 0em
                       \hfil
                       \parbox{\dimen0}{\def\baselinestretch{0.96}\small
                               {\bf #1.} #2
                               }
                       \hfil}
       \else \hbox to \hsize{\hfil \box\@tempboxa \hfil}
       \fi
       }
\makeatother

\newcommand{\EE}{\ensuremath{{\mathbb{E}}}}

\newcommand{\R}{\mathbb R}

\newcommand{\sign}{\operatorname{\sf{sign}}}

\newcommand{\p}{\operatorname{\mathbbm{P}}}
\newcommand{\E}{\operatorname{\mathbbm{E}}}

\newcommand{\cB}{\mathcal{B}}

\newcommand{\cR}{\mathcal{R}}

\newcommand{\cE}{\mathcal{E}}
\newcommand{\cN}{\mathcal{N}}

\newcommand{\cU}{\mathcal{U}}

\newcommand{\cH}{\mathcal{H}}

\newcommand{\eps}{\epsilon}

\newcommand{\bone}{\mathbf{1}}

\newcommand{\bbone}{\mathbbm{1}}

\newcommand{\Bin}{\operatorname{\mathsf{Bin}}}

\newcommand{\subG}{\operatorname{\mathsf{subG}}}
\newcommand{\tr}{\operatorname{\mathsf{tr}}}

\newcommand{\Var}{\operatorname{\mathsf{Var}}}

\begin{document}

\begin{center}

{\bf{\LARGE{Sharp analysis of EM for learning mixtures of \\pairwise differences}}}

\vspace*{.2in}

{\large{
\begin{tabular}{ccc}
Abhishek Dhawan$^\star$, Cheng Mao$^\star$, Ashwin Pananjady$^{\dagger, \ddagger}$
\end{tabular}
}}
\vspace*{.2in}

\begin{tabular}{c}
$^\star$School of Mathematics \\
$^\dagger$School of Industrial and Systems Engineering \\
$^\ddagger$School of Electrical and Computer Engineering \\
Georgia Institute of Technology
\end{tabular}

\vspace*{.2in}

\today

\vspace*{.2in}

\begin{abstract}
We consider a symmetric mixture of linear regressions with random samples from the pairwise comparison design, which can be seen as a noisy version of a type of Euclidean distance geometry problem. We analyze the expectation-maximization (EM) algorithm locally around the ground truth and establish that the sequence converges linearly, providing an $\ell_\infty$-norm guarantee on the estimation error of the iterates. Furthermore, we show that the limit of the EM sequence achieves the sharp rate of estimation in the $\ell_2$-norm, matching the information-theoretically optimal constant. We also argue through simulation that convergence from a random initialization is much more delicate in this setting, and does not appear to occur in general. Our results show that the EM algorithm can exhibit several unique behaviors when the covariate distribution is suitably structured.
\end{abstract}
\end{center}

\tableofcontents







\section{Introduction}
\emph{Mixtures of linear regressions} are classical methods used to model heterogeneous populations~\citep{jordan1994hierarchical,xu1994alternative,viele2002modeling} and have been widely studied in recent years (see, e.g.,~\citet{chaganty2013spectral,chen2017convex,balakrishnan2017,li2018learning,kwon2019global,chen2020learning}). A common approach is to apply a spectral algorithm to initialize the parameters of interest, and then run a \emph{locally convergent} nonconvex optimization algorithm; alternatively, one could even run the nonconvex algorithm from a random initialization.
These nonconvex optimization algorithms have been extensively analyzed under Gaussian assumptions on the covariates (see, e.g.~\citet{balakrishnan2017,klusowski2019estimating,chen2019gradient,kwon2020converges}), and it is known that they can exhibit favorable behavior globally in some settings.

In many tasks such as ranking~\citep{bradley1952rank}, crowd-labeling~\citep{dawid1979maximum} and web search~\citep{chen2013pairwise}, however, measurements are taken as \emph{pairwise comparisons} between entities, which renders the covariates inherently \emph{discrete}. These designs or covariates are far from Gaussian, and it is natural to ask how standard nonconvex optimization algorithms now perform. Motivated by this question, we study how the celebrated expectation-maximization (EM) algorithm behaves when estimating a mixture of symmetric linear regressions under the pairwise comparison design.
More specifically, we consider the following model of a symmetric mixture of two linear regressions with parameter vectors $\theta^*, - \theta^* \in \R^d$.
For $r = 1, \dots, N$, suppose that we observe $(x_r, y_r) \in \R^d \times \R$ satisfying 
\begin{align}\label{eq:model_symmetric}
y_r = z_r \, x_r^\top \theta^* + \epsilon_r .
\end{align} 
Here the covariates $x_r$ follow the pairwise comparison design, i.e.,
$x_r = e_{i_r}-e_{j_r}$ for $i_r, j_r \in [d] := \{1,\dots,d\}$, where $e_i$ is the $i$th standard basis vector in $\R^d$;
%
$z_r \in \{-1, 1\}$ is a latent sign indicating whether the observation is from component $\theta^*$ or $-\theta^*$;
%
$\epsilon_r$ models additive noise, assumed to be i.i.d.\ Gaussian for theoretical results.
Our goal is to estimate the parameter vector $\theta^* \in \R^d$.

The model~\eqref{eq:model_symmetric} can also be seen as a special case of the \emph{Euclidean distance geometry problem}~\citep{gower1982euclidean,liberti2014euclidean} and multidimensional scaling~\citep{borg2005modern}, and bears some resemblance to real phase retrieval (see~\citet{eldar2014phase,chen2019gradient} and the references therein). 
Let $y'_r = |y_r|$ be the observation and $\eps'_r = z_r \eps_r$ denote noise.
Because the sign $z_r$ is unobserved,~\eqref{eq:model_symmetric} is equivalent to 
\begin{align} \label{eq:model_abs}
y'_r = |x_r^\top \theta^* + \eps'_r|. 
\end{align}
Since $x_r = e_{i_r} - e_{j_r}$, we observe $|\theta^*_{i_r} - \theta^*_{j_r} + \eps'_r|$ which is the distance between two parameters contaminated by noise\footnote{It is more natural to have the noise $\eps'_r$ inside the absolute value because the distance is nonnegative.}. 
Hence our goal is to reconstruct $\theta^*_1, \dots, \theta^*_d$ from pairwise distances.
This is a noisy version of the classical Euclidean distance geometry problem, also known as multidimensional scaling (although the latter is more often used for a class of visualization methods). 
The above model is also related to the problem of angular or phase synchronization \citep{singer2011angular,zhong2018near,gao2021exact}, where angles $\theta^*_1, \dots, \theta^*_d \in [0, 2 \pi)$ are estimated from noisy measurements of $\theta^*_i - \theta^*_j$ mod $2 \pi$.
Moreover, a heterogeneous version of angular synchronization \citep{cucuringu2020extension} aims to reconstruct multiple groups of angles from a mixture of pairwise differences between angles from unknown groups. For all the aforementioned observation models, researchers have considered convex, spectral, and nonconvex fitting approaches.

As mentioned previously, our focus is on the EM algorithm~\citep{dempster1977maximum}---and we discuss other estimation procedures briefly in Section~\ref{sec:discuss}---which has long been used for estimation in latent variable models. Asymptotic convergence guarantees for EM to local optima in general latent variable models are classical~\citep{wu1983convergence}.
For Gaussian covariates $x_r$ in \eqref{eq:model_symmetric}, nonasymptotic, local linear convergence of the EM algorithm for symmetric mixtures of linear regressions was established in~\cite{balakrishnan2017}. In the same setting, a form of ``global convergence"---where the EM algorithm is first initialized with a close relative called ``Easy-EM'' which is in itself randomly initialized (see the background section to follow)---was established in \cite{kwon2019global}. There is a large literature on EM and other iterative algorithms for mixtures of regressions and related models; see, e.g., \cite{daskalakis2017ten,xu2016global,li2018learning,klusowski2019estimating,wu2019randomly,dwivedi2020singularity,kwon2020converges,kwon2021minimax} and references therein. EM has also been analyzed in progressively more complex settings, most recently in Gaussian latent tree models~\citep{pmlr-v178-dagan22b}.
Our paper should be seen, in spirit, as extending this line of investigation.

\subsection{Contributions and organization}

In this paper, we apply the EM algorithm to model \eqref{eq:model_symmetric} and establish its local linear convergence around the ground truth $\theta^*$, proving that once the algorithm is initialized sufficiently close to the ground truth $\theta^*$, it converges linearly fast to a fixed point $\hat \theta$. Moreover, we provide optimal bounds on the $\ell_\infty$ and $\ell_2$ estimation errors achieved by the EM fixed point in estimating $\theta^*$. Despite the extensive study of the EM algorithm for mixture models, our results differ from existing works in several crucial ways.

First, we establish \emph{entrywise} guarantees on the EM algorithm---our results hold in the stronger $\ell_\infty$-norm, not just the $\ell_2$-norm. Second, and in contrast to several results in the Gaussian case for mixtures of linear regressions~\citep{kwon2019global,chandrasekher2021sharp}, we do not assume that each iterate of the EM algorithm takes in a fresh sample, i.e., our convergence guarantee is not based on sample-splitting\footnote{Some results of~\cite{balakrishnan2017} do not assume sample-splitting but are suboptimal in the low-noise regime.}. Third, our result for the $\ell_2$-norm is \emph{sharp} in that it also is optimal in terms of the constant factor, while most previous results for EM only achieve the optimal rate up to constant factors. Fourth, our simulation results show that EM's convergence from a random initialization is quite delicate: convergence does not occur in general and this stands in sharp contrast to the case with Gaussian covariates. On a technical level, our results are enabled by analyzing the finite-sample EM operator \emph{directly}, whereas many previous results proceed through the population update. En route to establishing our results, we analyze the $\ell_{\infty \to \infty}$ operator norm of the psedoinverse of the sample covariance in this problem, a result that may be of independent interest (see Section~\ref{sec:proof-tech}).

The rest of this paper is organized as follows. In Section~\ref{sec:background}, we formally introduce our assumptions and the EM iteration that we study. In Section~\ref{sec:main-results}, we state and discuss the main theorems proved in this paper.
In Section~\ref{sec:discuss}, we present a host of experiments showing that the conditions appearing in our theorem statements are indeed necessary in some respect. These experiments also confirm that global convergence of EM in this setting is a delicate phenomenon and does not appear to occur in general. Section~\ref{sec:discuss} also discusses other algorithms and the related literature, and discusses the (non-)identifiability of mixtures with more than two components. We conclude with a sketch of the proof techniques in Section~\ref{sec:proof-tech}. Full proofs can be found in Section~\ref{sec:analysis}.

\subsection{Notation} \label{sec:notation}
Let $d$ and $N$ be positive integers such that $3 \le d \le N$ throughout the paper. 
Define $[d] := \{1,\dots,d\}$ and $\binom{[d]}{2} := \{(i,j) : i,j \in [d], \, i<j\}$. 
Let $\bone$ denote the all-ones vector in $\R^d$, and let $\cH$ denote the orthogonal complement of $\bone$ in $\R^d$.
We use the standard asymptotic notation $O(\cdot)$ and $o(\cdot)$ as $N \to \infty$; we use $O_p(\cdot)$ and $o_p(\cdot)$ when the asymptotic relation holds between two random variables with high probability.
We also write $a_N \ll b_N$ if $a_N = o(b_N)$, $a_N \gg b_N$ if $b_N = o(a_N)$, and $a_N \lesssim b_N$ if $a_N = O(b_N)$.
For a matrix $A \in \R^{d \times d}$, we denote the trace of $A$ by $\tr(A)$.

\section{Background and problem formulation} \label{sec:background}
For the results to follow, we consider model \eqref{eq:model_symmetric} with the following assumptions.

\begin{assumption}
\label{def:main-assumptions}
In \eqref{eq:model_symmetric}, the covariates are $x_r = e_{i_r}-e_{j_r}$ where $(i_r, j_r)_{r=1}^N$ are i.i.d.\ and uniformly random in $\binom{[d]}{2} = \{(i,j) : i,j \in [d], \, i<j\}.$
The noise terms $(\epsilon_r)_{r=1}^N$ are i.i.d.\ $\cN(0,\sigma^2)$ and independent from the covariates.
Moreover, $\mathbf{1}^\top \theta^* = 0$, i.e., $\theta^* \in \cH$ where $\cH$ is the orthogonal complement of the all-ones vector $\bone$ in $\R^d$.
\end{assumption}
\noindent
The assumption on the covariates states that the comparisons are chosen uniformly at random and with replacement among all pairs of indices.
There is no loss of generality in assuming $\bone^\top \theta^* = 0$, because the model \eqref{eq:model_symmetric} is invariant under any constant shift of $\theta^*$: If $\theta^*$ is replaced by $\theta^* + c\,\mathbf{1}$ for a constant $c \in \R$, the response $y_r$ stays the same. 
%
%
%
%
%
Define the sample covariance matrix of the covariates to be
\begin{equation}
\hat \Sigma := \frac{d-1}{2N} \sum_{r = 1}^N x_r x_r^\top.
\label{eq:def-samp-cov}
\end{equation}
Note that $\bone$ is in the kernel of $\hat \Sigma$ and $\cH$ is an invariant subspace of $\hat \Sigma$.
Therefore, $\hat \Sigma$ can be viewed as a linear map on $\cH$.
Let $\hat \Sigma^\dag$ denote the pseudoinverse of $\hat \Sigma$; it is the true inverse of $\hat \Sigma$ if $\hat \Sigma$ is restricted to $\cH$ and is invertible.
Moreover, the normalization $\frac{d-1}{2N}$ is chosen so that $\E[\hat \Sigma] \in \R^{d \times d}$ has entries
$$
(\E[\hat \Sigma])_{ij} = 
\begin{cases}
(d-1)/d & \text{ if } i = j , \\
-1/d & \text{ if } i \ne j.
\end{cases}
$$
It is not hard to check that $\E[\hat \Sigma]$ is the identity map on $\cH$.

It is convenient to define the so-called ``Easy-EM" operator $\bar Q: \cH \to \cH$~\citep{kwon2019global} and write the EM operator $\hat Q: \cH \to \cH$ in terms of it:
\begin{subequations} \label{eq:qq}
\begin{align}
\bar Q(\theta) &:= \frac{d-1}{2N} \sum_{r=1}^N \tanh \Big( \frac{ y_r \, x_r^\top \theta }{\sigma^2} \Big) y_r x_r, \label{eq:qq-1} \\
\hat Q(\theta)  &:=  \hat \Sigma^\dag \bar Q(\theta).
\label{eq:qq-2}
\end{align}
\end{subequations}
See \citet{balakrishnan2017} for a derivation of the EM operator.
Starting from an initialization $\theta^{(0)}$, the EM iteration is then given by
\begin{equation}
\theta^{(t+1)} := \hat Q( \theta^{(t)} ) = \bigg(\sum_{r=1}^Nx_rx_r^\top\bigg)^\dag  \sum_{r=1}^N \tanh \Big( \frac{ y_r \, x_r^\top \theta^{(t)} }{\sigma^2} \Big) y_r x_r , \qquad t \ge 0.
\label{eq:iteration-em}
\end{equation}

\section{Main results} \label{sec:main-results}

It is helpful to define the following class of parameter vectors. For a constant $\beta > 0$, let
\begin{equation}
\Theta(\beta)
:= \bigg\{ \theta \in \cH : \theta_1 \le \cdots \le \theta_d , \, |\theta_i - \theta_j| \ge \beta \frac{|i-j|}{d} \text{ for any } (i,j) \in \binom{[d]}{2} \bigg\} .
\label{eq:simple-parameter-set}
\end{equation}
Operationally, the set $\Theta(\beta)$ is a natural family of parameters to consider. First, up to a relabeling of the coordinates, there is no loss of generality in restricting our attention to parameter vectors with nondecreasing entries. 
Second, each parameter vector in $\Theta(\beta)$ has its $i$th and $j$th entries separated by at least $\beta \frac{|i-j|}{d}$.
This separation condition is imposed to simplify the statement of our main results.
It also appeared in, e.g., \cite{chen2022optimal}, where the pairwise comparison design was studied.
Our full results in Section~\ref{sec:analysis} are more general, assuming a condition weaker than that in \eqref{eq:simple-parameter-set} (see \eqref{eq:cond-2-s-i-delta} in particular).
In short, for every $i \in [d]$, we require the set $\{j \in [d] : |\theta_i - \theta_j| \le c_1 \beta \}$ to contain at most $c_2 \, d$ elements for some constants $c_1, c_2 > 0$; in other words, $\theta_i$ should not have too many ``neighbors'' $\theta_j$.
This weaker condition is satisfied with high probability if $\theta^*_i$ are i.i.d.\ uniform random variables in $[-1,1]$, for example.

With this setup in hand, we are now ready to state our main results. Our first result shows that the EM sequence $\{\theta^{(t)}\}_{t \ge 0}$ converges linearly to a limit $\hat \theta$ around the ground truth $\theta^*$, and provides an estimation error guarantee in the $\ell_\infty$-norm with high probability.

\begin{theorem}
\label{thm:em-contraction}
Consider the model in \eqref{eq:model_symmetric} satisfying Assumption~\ref{def:main-assumptions}.
For any fixed constants $\rho, D > 0$, there exist constants $N_0, c_1, C_2 > 0$ depending only on $\rho$ and $D$ such that the following holds.
For $\beta > 0$, suppose $\theta^* \in \Theta(\beta)$ as defined in \eqref{eq:simple-parameter-set}.
Suppose $N \ge \max\{d^{1 + \rho}, N_0\}$ and $\sigma \le c_1 \beta$. 
Fix $\theta^{(0)} \in \cH$ such that $\|\theta^{(0)} - \theta^*\|_\infty \le c_1 \beta$.
Let $\{\theta^{(t)}\}_{t \ge 0}$ be the EM iterates defined in \eqref{eq:iteration-em}. 
Moreover, let 
$$
\tau := C_2 \sigma \sqrt{\frac{d}{N}\log N} , \qquad
T := \max\bigg\{ 0, \bigg\lceil \log_{4/3} \bigg( \frac{\|\theta^{(0)} - \theta^*\|_\infty}{4\tau} \bigg) \bigg\rceil \bigg\} . 
$$
Then it holds with probability at least $1-N^{-D}$ that 
\begin{itemize}
\item
there exists $\hat \theta \in \cH$ such that $\hat Q(\hat \theta) = \hat \theta$ and $\|\hat \theta - \theta^*\|_\infty \le 4 \tau$;

\item
the sequence $\{\theta^{(t)}\}_{t \ge 0}$ converges to $\hat \theta$;

\item
$\| \theta^{(T+t)} - \hat \theta \, \|_\infty \le 8 \tau / 2^t$ for all $t \ge 0$.
\end{itemize}
\end{theorem}

Theorem~\ref{thm:em-contraction} is proved in Section~\ref{sec:pf-main-results} and is a result of Proposition~\ref{prop:em-contraction}.
We begin with a discussion of the conditions in the theorem.
First, the setting is nearly high-dimensional because we only require $N \ge d^{1+\rho}$ for an arbitrarily small $\rho > 0$.
We remark that it is possible to improve this condition to $N \ge d \cdot \operatorname{polylog}(d)$ by a more careful application of Proposition~\ref{prop:em-contraction} at the cost of complicating the conditions on the other parameters (see Section~\ref{sec:d-polylog-d} for details).
Focusing on the setting where $\beta$ is a constant for clarity, we have $|\theta^*_i - \theta^*_j| \asymp |i-j|/d$, and $\theta^*_i = O(1)$ for all $i, j \in [d]$. 
The constant-sized noise condition $\sigma \le c_1 \beta$ is mild, because it should be compared against the minimum signal strength $|\theta^*_{i} - \theta^*_{i+1}| \asymp 1/d$ (see Eq.~\eqref{eq:simple-parameter-set}).
Next, the initialization $\theta^{(0)}$ has to be close to the true parameter vector $\theta^*$ so that $\|\theta^{(0)} - \theta^*\|_\infty \le c_1 \beta$. This is the sense in which the convergence guarantee is local, and we conjecture that some such condition is necessary (see Section~\ref{sec:discuss} for experimental verification of this conjecture).

With these conditions assumed, we now explain the conclusions of Theorem~\ref{thm:em-contraction}.
It is useful to note that $\tau \asymp \sigma \sqrt{\frac{d}{N}\log N}$ is the optimal rate\footnote{In fact, this is the optimal rate even if we consider linear regression without a mixture, i.e., where all the signs $z_r$ are known a priori.} we expect to have when estimating each entry $\theta^*_i$ of a $d$-dimensional vector from $N$ samples with probability $1 - N^{-D}$.
It is easiest to interpret Theorem~\ref{thm:em-contraction} as a two-stage result for the EM iteration:
\begin{itemize}
\item
First, the quantity $T$ indicates that it takes at most a logarithmic number of steps for the EM iteration to get from distance $\|\theta^{(0)} - \theta^*\|_\infty$ to an $O(\tau)$-neighborhood around the ground truth $\theta^*$.
By the bounds $\|\hat \theta - \theta^*\|_\infty \le 4 \tau$ and $\|\theta^{(T)} - \hat \theta\|_\infty \le 8 \tau$ (by taking $t = 0$ in the last statement), we have $\|\theta^{(T)} - \theta^*\|_\infty = O(\tau)$.
Hence, $\theta^{(T)}$ already achieves the optimal rate of estimation in the $\ell_\infty$-norm.

\item
Second, beyond time $T$, the EM sequence $\{\theta^{(t)}\}_{t \ge 0}$ eventually converges to a limit $\hat \theta$, and the rate of convergence is linear.
In other words, the EM operator is locally contractive.
Moreover, the limit $\hat \theta$ lies in the $4 \tau$-neighborhood of $\theta^*$ in the $\ell_\infty$-norm.
\end{itemize}
Finally, we once again remark that the theorem does not require any resampling for the EM iteration.
The conclusions of the theorem hold simultaneously on a high-probability event.

Our second theorem shows that the limit $\hat \theta$ of the EM sequence achieves the sharp estimation error in the $\ell_2$-norm with high probability in the low-noise regime, in that the constant is also sharp.

\begin{theorem}\label{thm:asymptotics}
In the setting of Theorem~\ref{thm:em-contraction} (which in particular assumes $N \ge d^{1 + \rho}$ for fixed $\rho > 0$), we additionally assume $\sigma \ll \frac{\beta}{(\log N)^{2}}$ and $d \gg \log N$.
Then it holds with probability at least $1 - N^{-D}$ that
$$
\|\hat \theta - \theta^*\|_2^2 \le (1+o_p(1)) \, \sigma^2 \frac{d-1}{2N} \tr(\hat \Sigma^\dag) .
$$
\end{theorem}

The proof of Theorem~\ref{thm:asymptotics} is deferred to Section~\ref{sec:pf-main-results}.
Compared to the noise condition $\sigma \le c_1 \beta$ in Theorem~\ref{thm:em-contraction}, here we need $\sigma$ to be smaller by a polylogarithmic factor, and an additional, mild condition on the dimension.
While we assume a random design of $x_r$ and thus the right-hand side of the above equation is random, the result can be understood as follows.
Conditional on a typical realization of the covariates $(x_r)_{r=1}^N$, the squared error $\|\hat \theta - \theta^*\|_2^2$ is within a $(1+o(1))$ factor of
\begin{equation}
\sigma^2 \frac{d-1}{2N} \tr(\hat \Sigma^\dag) = \sigma^2 \tr\bigg( \Big( \sum_{r=1}^N x_r x_r^\top \Big)^\dag \bigg) 
\label{eq:sharp-rate}
\end{equation}
with high probability over the randomness of the noise $(\eps_r)_{r=1}^N$.
Note that \eqref{eq:sharp-rate} is the optimal rate even for vanilla linear regression with a fixed design (see Section~\ref{sec:lin-reg-rate} for more discussion). 
As a result, we obtain the optimal $\ell_2$ error (including the sharp constant) for the EM algorithm provided that the initialization is sufficiently informative and the noise level is sufficiently small.

\section{Experiments and discussion}
\label{sec:discuss}
In this section, we further discuss our model and results, related methods, and possible extensions.
Our discussion is accompanied by numerical experiments.


\subsection{Failure of random initialization}
Theorem~\ref{thm:em-contraction} requires that the initialization of the EM algorithm satisfies $\|\theta^{(0)} - \theta^*\|_\infty \le c \beta$ for a constant $c>0$.
We believe that this is not an artifact of the analysis and that some such condition is necessary. 
To test this hypothesis, let us numerically test if a ``random'' initialization suffices for the convergence of EM.


For simplicity, we take $\theta^*_i = \frac{i}{d} - \frac{d+1}{2 d}$ for $i \in [d]$ so that $|\theta^*_i - \theta^*_j| = \frac{|i-j|}{d}$.
Then $\theta^* \in \Theta(\beta)$ with $\beta = 1$.
Let $\theta \in \R^d$ be a random vector whose entries are i.i.d.\ uniform in $[-0.5, 0.5]$; then define $\theta^R = \theta - a \bone$ where the scalar $a \in \R$ is chosen so that $\bone^\top \theta^R = 0$.
The vector $\theta^R$ is a canonical random initialization, because the entries of both $\theta^R$ and $\theta^*$ are approximately in the range $[-0.5, 0.5]$, yet $\theta^R$ is random.
For $\eta \in [0,1]$, define
$
\theta^{(0)} = \theta^{(0)} (\eta) = (1-\eta) \theta^* + \eta \, \theta^R ,
$
which interpolates between the ground truth $\theta^*$ and the random vector $\theta^R$.
We use $\theta^{(0)}$ as the initialization of the EM algorithm and study its performance.
Note that 
$$
\|\theta^{(0)} - \theta^*\|_\infty = \eta \, \|\theta^R - \theta^*\|_\infty 
$$
which is approximately $\eta$.
Our theory predicts that the EM algorithm performs well when $\eta$ is a small constant.
On the other hand, if a random initialization is not sufficient, then the EM algorithm will have poor performance when $\eta$ is close to $1$.
We indeed observe this phenomenon in the left plot of Figure~\ref{fig:initialization}.

\begin{figure}[ht!]
\centering
\includegraphics[clip, trim=1.6cm 6.4cm 2.4cm 7.4cm, width=0.45\textwidth]{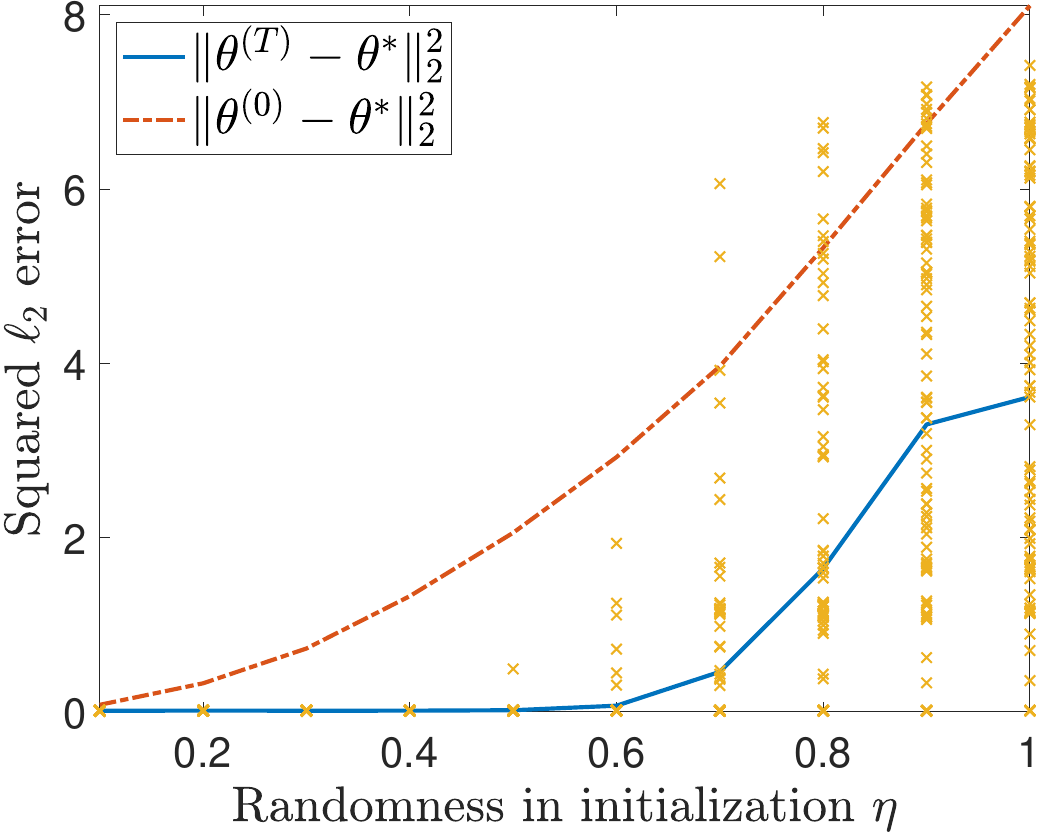}
\hspace{0.8cm}
\includegraphics[clip, trim=1.6cm 6.4cm 2.4cm 7.4cm, width=0.45\textwidth]{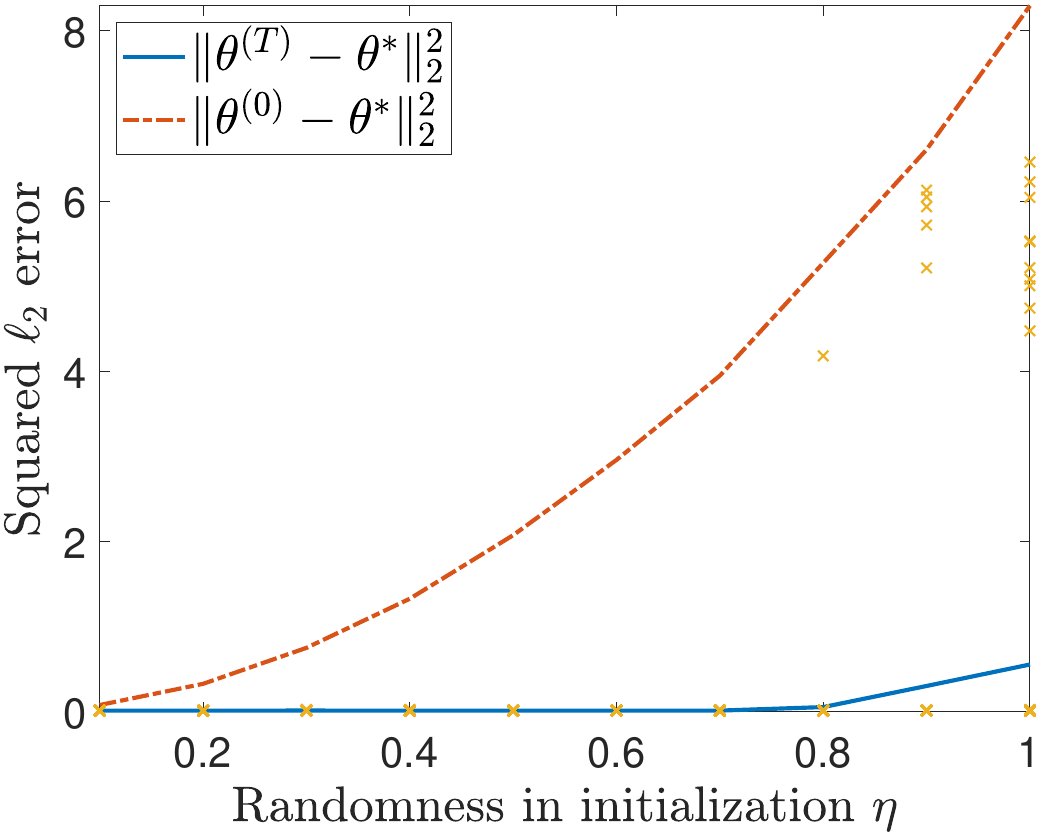}
%
%
\caption{\emph{Left:} Pairwise-comparison design. \emph{Right:} Gaussian design. We plot the initial squared error $\|\theta^{(0)} - \theta^*\|_2^2$ (dashed red line) and the eventual squared error $\|\theta^{(T)} - \theta^*\|_2^2$ (solid blue line) of the EM algorithm against a varying randomness parameter $\eta$ in the initialization. 
The experiment for each value of $\eta$ is averaged over $100$ repetitions, and the error $\|\theta^{(T)} - \theta^*\|_2^2$ for each repetition is indicated by a yellow cross.
}
\label{fig:initialization}
\end{figure}

To be more precise, we set $d = 50$, $N = 1000$, and $\sigma = 0.1$.
For $\eta \in \{0.1, 0.2, \dots, 1\}$, we start from $\theta^{(0)} = \theta^{(0)}(\eta)$ and run the EM algorithm for $T = 100$ steps\footnote{The EM algorithm typically converges within a few steps, especially when the initialization is close to the ground truth, but we make the conservative choice $T = 100$ to ensure convergence.}.
We plot the initial squared error $\|\theta^{(0)} - \theta^*\|_2^2$ (dashed red line) and the eventual squared error $\|\theta^{(T)} - \theta^*\|_2^2$ (solid blue line)\footnote{Since the model \eqref{eq:model_symmetric} is invariant if $\theta^*$ is replaced by $-\theta^*$, convergence to a neighborhood of $-\theta^*$ should also be considered as a success. Hence, we actually compute the error $\min \{\|\theta^{(T)} - \theta^*\|_2^2, \|\theta^{(T)} + \theta^*\|_2^2\}$ in all our experiments. For brevity, we do no mention this elsewhere.}, averaged over $100$ independent repetitions for each value of $\eta$.
The error $\|\theta^{(T)} - \theta^*\|_2^2$ for each repetition is indicated by a yellow cross.
As shown by the left plot of Figure~\ref{fig:initialization}, an initialization with $\eta  \le 0.5$ leads to an extremely small error $\|\theta^{(T)} - \theta^*\|_2^2$ with high probability, while for an initialization with a larger $\eta$, the EM algorithm fails to achieve a small error with constant probability.

The failure of random initialization in the setting of the pairwise comparison design stands in sharp contrast to its behavior for Gaussian designs, 
as confirmed by the right plot of Figure~\ref{fig:initialization}.
The setup and parameters of this experiment are exactly the same as before, except that we now take i.i.d.\ covariates $x_r \sim \cN(0, \frac{2}{d} I)$, where the normalization $\frac 2d$ is chosen to match the scaling of $x_r$ in Assumption~\ref{def:main-assumptions}.
As we see in the figure, the EM algorithm achieves a small error $\|\theta^{(T)} - \theta^*\|_2^2$ with high probability\footnote{The occasional failures of random initialization with $\eta = 1$ can be explained by a few factors, such as our moderate choices of $d$ and $N$, the nontrivial error probability, etc.} even if $\eta$ is close to $1$. 
We remark that there is a sizable literature on the success of iterative algorithms with random initialization for a variety of problems, such as phase retrieval \citep{chen2019gradient}, Gaussian mixtures \citep{dwivedi2020singularity,wu2019randomly}, mixtures of log-concave distributions \citep{qian2019global}, and general regression models with Gaussian covariates \citep{chandrasekher2021sharp}.
However, Gaussianity or continuous density is typically part of the assumption, and analyzing iterative algorithms beyond the Gaussian setting appears to be a generally more challenging problem.
See also \cite{dudeja2022universality,wang2022universality} for recent work in this direction in the framework of approximate message passing.

An interesting open problem is to offer an explanation for the contrasting behavior of the randomly initialized EM algorithm given the two types of covariates.
We speculate that random initialization tends to fail if the covariates are discrete and sparse.
Note that our covariates $x_r$ are supported in the highly structured set $\{e_i - e_j : (i,j) \in \binom{[d]}{2}\}$, unlike the Gaussian covariates which are in general positions in $\mathbb{R}^d$.
However, we do not have a theoretical explanation for this phenomenon and leave the question to future work.

\subsection{Spectral method, initialization, and comparison}
Spectral methods form another popular class of algorithms for tackling mixture models and the Euclidean distance geometry problem.
One canonical spectral method for localizing points given their pairwise distances is the classical multidimensional scaling \citep{borg2005modern}.
It takes the following form for model \eqref{eq:model_symmetric}.
Define a matrix $D \in \R^{d \times d}$ that stores noisy pairwise distances by
$$
D_{ij} = D_{ji} = \frac{d(d-1)}{2N} \sum_{r=1}^N (y_r^2 - \sigma^2) \, \bbone\{ x_r = e_i - e_j \} .
$$
Since 
$\E[y_r^2 - \sigma^2 \mid x_r = e_i - e_j]
= \E[ (\theta^*_i - \theta^*_j)^2 + 2 \eps_r z_r (\theta^*_i - \theta^*_j) + \eps_r^2 - \sigma^2 \mid x_r ]
= (\theta^*_i - \theta^*_j)^2 ,$
we see that
$$
\E[D_{ij}] = (\theta^*_i - \theta^*_j)^2 .
$$
In other words, $\E[D]$ stores the pairwise distances between the entries of $\theta^*$.
Let $J = I - \frac{1}{d} \bone \bone^\top$ where $I$ is the identity matrix in $\R^{d \times d}$ and $\bone$ is the all-ones vector in $\R^d$.
Using that $\E[D]$ is the pairwise distance matrix, it is not hard to verify that
$$
- \frac 12 J \E[D] J = \theta^* (\theta^*)^\top ,
$$
which is the Gram matrix of $\theta^*$.
Therefore, if $(\lambda_1, v_1)$ is the leading eigenpair of the matrix $- \frac 12 J D J$, then we can use $\tilde \theta = \sqrt{\lambda_1} \, v_1$ as an estimator of $\theta^*$.

Using standard spectral perturbation theory and random matrix theory, it is straightforward to analyze the error $\tilde \theta - \theta^*$ in either the $\ell_2$ or $\ell_\infty$ norm (see, e.g., \cite{chen2021spectral}).
We do not provide a theoretical analysis of $\tilde \theta$ in this work, but let us discuss it qualitatively.
There are two sources of error in the spectral estimator $\tilde \theta$: the noise $\eps_r$, which is inevitable, and the incompleteness of observations.
In the regime $N < \binom{d}{2}$, if we never have $x_r = e_i - e_j$ for a pair of indices $(i,j)$, then $D_{ij} = 0$ while $\E[D_{ij}] = (\theta^*_i - \theta^*_j)^2$.
As a result, the spectral estimator $\tilde \theta$ does not recover $\theta^*$ even if $\sigma \to 0$.
On the other hand, the limit of the EM algorithm $\hat \theta$ recovers $\theta^*$ as shown by Theorem~\ref{thm:asymptotics} in this regime.
Nevertheless, this does not mean that the spectral estimator is useless, because it can be used as an initialization of the EM algorithm.
The experiment exhibited in the left plot of Figure~\ref{fig:spectral} confirms these observations.

\begin{figure}[ht]
\centering
\includegraphics[clip, trim=0.7cm 6.4cm 1.8cm 7.3cm, width=0.45\textwidth]{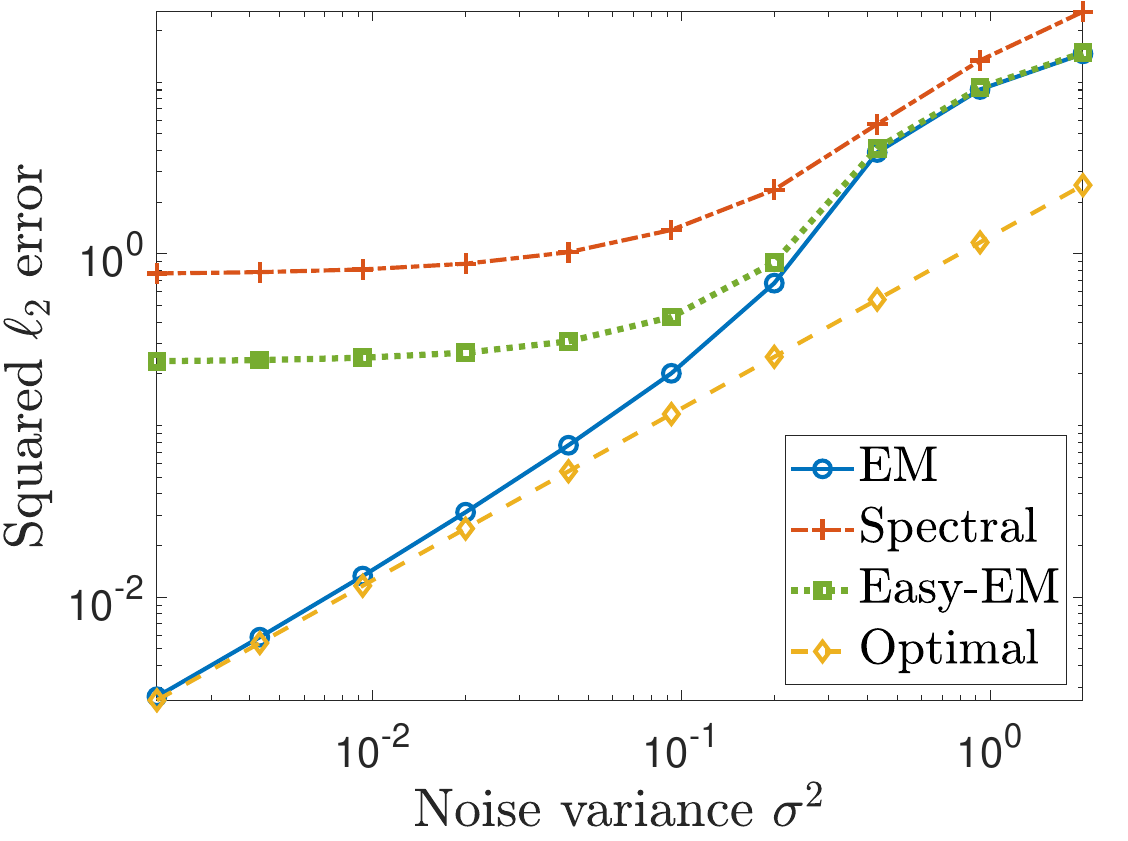}
\hspace{0.8cm}
\includegraphics[clip, trim=0.7cm 6.4cm 1.8cm 7.3cm, width=0.45\textwidth]{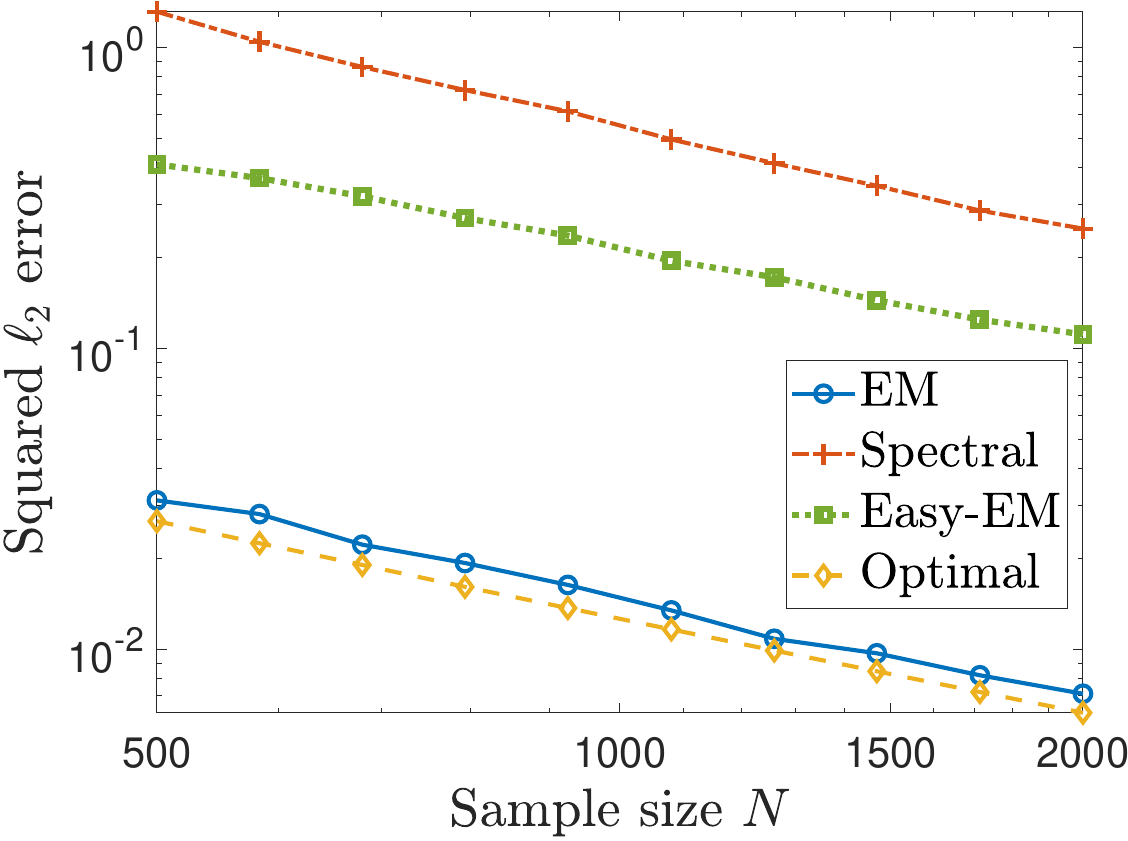}
\caption{\emph{Left:} Varying noise standard deviation.  \emph{Right:} Varying sample size. The log-log plots show the performance of different algorithms against the varying parameters. The squared error $\|\tilde \theta - \theta^*\|_2^2$ for the spectral estimator is shown in dash-dotted red lines. The EM and Easy-EM algorithms are initialized with $\theta^{(0)} = \tilde \theta$ and run for $T = 20$ steps. Their eventual squared errors $\|\theta^{(T)} - \theta^*\|_2^2$ are shown in solid blue lines and dotted green lines respectively. The experiment for each value of $\sigma^2$ or $N$ is averaged over $100$ independent repetitions. The theoretically predicted optimal rate is shown in dashed yellow lines.}
\label{fig:spectral}
\end{figure}

In the left plot of Figure~\ref{fig:spectral}, we set $d = 50$, $N = 1000$, and let the noise variance $\sigma^2$ vary from $0.002$ to $2$.
Let us first focus on the performance of estimators in the small-noise regime.
When $\sigma \to 0$, the squared error $\|\tilde \theta - \theta^*\|_2^2$ of the spectral method (dash-dotted red line) does not converge to zero.
However, we can set $\theta^{(0)} = \tilde \theta$ and run the EM algorithm from this initialization for $T = 20$ steps.
The squared error $\|\theta^{(T)} - \theta^*\|_2^2$ (solid blue line) goes to zero as $\sigma \to 0$.
Recall that Theorem~\ref{thm:asymptotics} predicts a sharp optimal rate \eqref{eq:sharp-rate} for the EM algorithm in the small-noise regime.
We also plot this optimal error (dashed yellow line) and observe that it is very close to the actual error $\|\theta^{(T)} - \theta^*\|_2^2$ as $\sigma \to 0$.
In addition, we consider the Easy-EM iteration with $\theta^{(t+1)} = \bar Q(\theta^{(t)})$ where $\bar Q$ is defined in \eqref{eq:qq-1}.
Compared to the EM iteration, Easy-EM does not have the multiplication by $\hat \Sigma^\dag$ at each step, which turns out to be crucial: We observe that the Easy-EM algorithm (dotted green line) fails to achieve a small error as $\sigma \to 0$.

Next, we consider the regime of large noise in the left plot of Figure~\ref{fig:spectral}.
The error of the EM algorithm no longer follows the sharp rate predicted by Theorem~\ref{thm:asymptotics}, which suggests that some condition on the noise may be necessary for this theorem (although the condition $\sigma \ll \frac{\beta}{(\log N)^2}$ we impose may not be optimal).
Nevertheless, the EM and Easy-EM algorithms still improve upon the spectral initialization, and all the three algorithms appear to differ only by a constant factor from the optimal rate as $\sigma^2$ grows.

To further investigate the rate of estimation achieved by these algorithms, we consider another experiment shown in the right plot of Figure~\ref{fig:spectral}.
We take $d = 50$, $\sigma = 0.1$, and let the sample size $N$ vary from $500$ to $2000$.
We again run the EM and the Easy-EM algorithm from the spectral initialization $\theta^{(0)} = \tilde \theta$ for $T = 20$ steps, and plot the squared errors.
Given that the noise is small, the optimal rate from Theorem~\ref{thm:asymptotics} accurately predicts the actual error of the EM algorithm.
The spectral estimator and the Easy-EM algorithm clearly have worse performance, but their rates of estimation appear to be optimal as $N$ grows.

\subsection{Extensions of the model}
There are several directions for extending our model.
For example, one may consider a mixture of $k$ linear regressions 
$$
y_r = x_r^\top \theta^{[\ell_r]} + \epsilon_r,
$$
where we have $\theta^{[1]},\dots, \theta^{[k]} \in \R^d$ and $(\ell_r)_{r=1}^N$ are i.i.d.\ from the distribution $\sum_{\ell=1}^k w_\ell \delta_\ell$ with weights satisfying $w_\ell \ge 0$ and $\sum_{\ell=1}^k w_\ell = 1$.
When the covariates $x_r$ are from a discrete distribution, even the identifiability of the mixture model is a nontrivial open problem.



To see the difficulty of the problem, let us consider a negative example.
Suppose that we have a mixture of $k=3$ equally weighted linear regressions, and the covariates are pairwise comparisons $x_r = e_{i_r} - e_{j_r}$ for $i_r, j_r \in [d]$.
For simplicity, suppose that the noise $\eps_r$ is zero.
Let
$$
\theta^{[1]} =
\begin{bmatrix}
1 \\ 2 \\ \theta_3 \\ \vdots \\ \theta_d
\end{bmatrix}
, \quad
\theta^{[2]} =
\begin{bmatrix}
3 \\ 3 \\ \theta_3 \\ \vdots \\ \theta_d
\end{bmatrix}
, \quad
\theta^{[3]} =
\begin{bmatrix}
2 \\ 4 \\ \theta_3 \\ \vdots \\ \theta_d
\end{bmatrix}
, \quad
\tilde{\theta}^{[1]} =
\begin{bmatrix}
2 \\ 2 \\ \theta_3 \\ \vdots \\ \theta_d
\end{bmatrix}
, \quad
\tilde{\theta}^{[2]} =
\begin{bmatrix}
1 \\ 3 \\ \theta_3 \\ \vdots \\ \theta_d
\end{bmatrix}
, \quad
\tilde{\theta}^{[3]} =
\begin{bmatrix}
3 \\ 4 \\ \theta_3 \\ \vdots \\ \theta_d
\end{bmatrix}
,
$$
where the unspecified entries $\theta_3, \dots, \theta_d$ can take any values.
It is not hard to see that, even if we observe the sets 
$$
\Big\{ \theta^{[\ell]}_i - \theta^{[\ell]}_j : \ell = 1,2,3 \Big\} \quad \text{ for all } (i, j) \in \binom{[d]}{2}
$$
with no noise, the two mixtures $\{ \theta^{[\ell]} : \ell = 1,2,3 \}$ and $\{ \tilde \theta^{[\ell]} : \ell = 1,2,3 \}$ yield the same sets of observations.
As a consequence, the mixture model with 3 components is not identifiable with the pairwise comparison design.

Therefore, to have an identifiable mixture model with multiple components, it is necessary to assume a richer design of covariates $x_r$.
Instead of pairwise comparisons, one may consider comparisons between multiple objects as in the Plackett--Luce model \citep{luce1959individual,plackett1975analysis} or groups of pairwise comparisons used for mixtures of permutations \citep{mao2022learning}.
We leave this interesting topic to future research.

Beyond a mixture of linear regressions with additive noise, one may consider a mixture of generalized linear regressions (see, e.g., \cite{heumann2008finite,sun2014learning,sedghi2016provable}), and in particular, the case of binary response is of practical interest in classification tasks.
Similar to the linear setting, existing theories for mixtures of generalized linear models are often based on a Gaussian or continuous design.
In the discrete setting, there have been a series of recent works on mixtures of the Placket--Luce or multinomial logit models motivated by ranking tasks (see, e.g., \cite{zhao2016learning,mollica2017bayesian,chierichetti2018learning,liu2019learning,zhao2022learning}).
In particular,~\citet{zhang2022identifiability} show that a mixture of two Bradley--Terry--Luce models (i.e., logistic regressions with the pairwise comparison design) are identifiable except when the parameters are in a set of measure zero.
Very recently,~\citet{nguyen2023efficient} provide a spectral initialization and an accurate implementation of the EM algorithm for learning a mixture of Plackett--Luce models. It would be interesting to analyze the convergence of iterative algorithms such as the EM algorithm for these models.

Along another direction, one may extend the model \eqref{eq:model_abs} to the noisy Euclidean distance geometry problem in dimension $m > 1$.
Namely, we aim to estimate vectors $\theta^*_1, \dots, \theta^*_d \in \R^{m}$ based on incomplete, noisy observations of pairwise distances $\| \theta^*_i - \theta^*_j \|_2$.
While this problem and its variants have long been studied (see the survey by \citet{liberti2014euclidean}), progress in understanding the sample complexity has been made only in recent years using tools of matrix completion \citep{lai2017solving,tasissa2018exact}, and the theory is even less complete in the noisy setting \citep{drusvyatskiy2017noisy,sremac2019noisy}.
Compared to semidefinite programs used in many of these works, iterative algorithms are faster to run, but remain to be understood for this problem.

\section{Proof techniques}
\label{sec:proof-tech}
We now informally discuss the strategies and key ingredients in the proofs of our main results; full proofs are provided in Section~\ref{sec:analysis}.

\subsection{Contraction in the $\ell_\infty$ norm}
\label{sec:contraction-discussion}
The first main result of this work is the linear convergence of the EM iteration in the $\ell_\infty$-norm in Theorem~\ref{thm:em-contraction}.
To shed light on this result, let us first consider the zero-noise limit of the EM algorithm for simplicity.
Suppose that the noise $\eps_r$ is zero in \eqref{eq:model_symmetric}.
Letting $\sigma \to 0$ in \eqref{eq:iteration-em}, we see that the EM iteration becomes
\begin{equation}
\theta^{(t+1)} = \bigg(\sum_{r=1}^Nx_rx_r^\top\bigg)^\dag \sum_{r=1}^N \sign(y_r \, x_r^\top \theta^{(t)}) \, y_r x_r
\label{eq:iteration-am}
\end{equation}
for $t \ge 0$, where $\sign(a) = 1$ if $a>0$, $\sign(a) = 0$ if $a=0$, and $\sign(a) = -1$ if $a<0$.
In fact, this is equivalent to the alternating minimization procedure where we iteratively compute
\begin{itemize}
\item
$z^{(t)}_r := \min_{z \in \{-1,1\}} (y_r - z \, x_r^\top \theta^{(t)})^2 = \sign(y_r \, x_r^\top \theta^{(t)})$ for $r \in [N]$;

\item
$\theta^{(t+1)} := \min_{\theta \in \cH} \sum_{r=1}^N (y_r - z^{(t)}_r x_r^\top \theta)^2 = \left(\sum_{r=1}^Nx_rx_r^\top\right)^\dag \sum_{r=1}^N z^{(t)}_r y_r x_r$.
\end{itemize}

We now explain why the iteration \eqref{eq:iteration-am} contracts to $\theta^*$ locally.
By slightly abusing the notation, we denote (only in this subsection) the noiseless Easy-EM and EM operators respectively by
$$
\bar Q(\theta) = \frac{d-1}{2N} \sum_{r=1}^N \sign(y_r \, x_r^\top \theta) \, y_r x_r, \qquad
\hat Q(\theta) = \hat \Sigma^\dag \bar Q(\theta) ,
$$
where the sample covariance $\hat \Sigma$ is defined in \eqref{eq:def-samp-cov}.
We first note that $\theta^*$ is a fixed point of the operator $\hat Q$.
Indeed, since $y_r = z_r x_r^\top \theta^*$, we have 
\begin{align*}
\hat Q(\theta^*) = \bigg(\sum_{r=1}^Nx_rx_r^\top\bigg)^\dag \sum_{r=1}^N \sign\left( z_r (x_r^\top \theta^*)^2 \right) \, z_r (x_r^\top \theta^*) x_r 
=  \bigg(\sum_{r=1}^Nx_rx_r^\top\bigg)^\dag \sum_{r=1}^N x_r x_r^\top \theta^* 
= \theta^*.
\end{align*}
Next, fix $\theta \in \cH$ in an $\ell_\infty$ neighborhood of $\theta^*$.
We have
$\hat Q(\theta) - \theta^*
= \hat \Sigma^\dag \left( \bar Q(\theta) - \bar Q(\theta^*) \right)$, so
\begin{equation}
\|\hat Q(\theta) - \theta^*\|_\infty
\le  \|\hat \Sigma^\dag\|_\infty \left\| \bar Q(\theta) - \bar Q(\theta^*) \right\|_\infty ,
\label{eq:infty-infty-bound}
\end{equation}
where $\|\hat \Sigma^\dag\|_\infty := \max_{v \in \cH: \|v\|_\infty = 1} \|\hat \Sigma^\dag v\|_\infty$.
It remains to show, for example, that for a quantity $L > 0$, we have
$$
\|\hat \Sigma^\dag\|_\infty \le L, \qquad
\left\| \bar Q(\theta) - \bar Q(\theta^*) \right\|_\infty \le \frac{1}{2L} ,
$$
so that we have the desired contraction $\|\hat Q(\theta) - \theta^*\|_\infty \le \frac 12 \|\theta - \theta^*\|_\infty$.

In fact, bounding $\|\hat \Sigma^\dag\|_\infty$ is one of the most important and technical steps of our proof.
We discuss this step in more detail in Section~\ref{sec:norm-inv-cov-discus} below and establish a bound formally in Proposition~\ref{prop:sigma-dag-bound}.
For the remaining portion of the proof, we must bound $\left\| \bar Q(\theta) - \bar Q(\theta^*) \right\|_\infty$. Note that
\begin{equation}
\bar Q(\theta) - \bar Q(\theta^*) = \frac{d-1}{2N} \sum_{r=1}^N \Big[ \sign \big( y_r x_r^\top \theta \big) - \sign \big( y_r x_r^\top \theta^* \big) \Big] y_r  x_r .
\label{eq:sign-sign}
\end{equation}
If $\theta$ is in an $\ell_\infty$ neighborhood of $\theta^*$, i.e., $\theta$ is close to $\theta^*$ entrywise, then $x_r^\top \theta = \theta_{i_r} - \theta_{j_r}$ is close to $x_r^\top \theta^* = \theta^*_{i_r} - \theta^*_{j_r}$.
As a result, $\sign \big( y_r x_r^\top \theta \big)$ will coincide with $\sign \big( y_r x_r^\top \theta^* \big)$ for most $r \in [N]$, and a majority of terms in the above sum will vanish.
Based on this intuition, we can perform an entrywise analysis of $\bar Q(\theta) - \bar Q(\theta^*)$ and obtain a bound on $\|\bar Q(\theta) - \bar Q(\theta^*)\|_\infty$.
We omit the details.

In the noisy setting, the strategy for analyzing the EM iteration \eqref{eq:iteration-em} is similar to the above noiseless analysis.
However, there are complications due to the presence of noise.
First, the fixed point of the EM operator $\hat Q$ is the random vector $\hat \theta$, not the deterministic vector $\theta^*$.
Therefore, we cannot directly show $Q(\hat \theta) = \hat \theta$; rather, it follows as a consequence of $\hat Q$ being a contraction locally around $\theta^*$.
Second, to prove that $\hat Q$ is contractive, we again reduce it to analyzing the Easy-EM operator $\bar Q$ via \eqref{eq:infty-infty-bound}, but now for $\theta, \theta' \in \cH$,
$$
\bar Q(\theta) - \bar Q(\theta') = \frac{d-1}{2N} \sum_{r=1}^N \bigg[ \tanh \Big( \frac{ y_r \, x_r^\top \theta }{\sigma^2} \Big) - \tanh \Big( \frac{ y_r \, x_r^\top \theta' }{\sigma^2} \Big) \bigg] y_r x_r .
$$
As the function $\tanh(\cdot)$ is a soft approximation of the $\sign(\cdot)$ function in \eqref{eq:sign-sign}, the intuition for why we can control $\bar Q(\theta) - \bar Q(\theta')$ remains the same, but a more quantitative analysis is necessary.
See Section~\ref{sec:convergence} for details.

\subsection{Expansion around the ground truth}

Our second main result is the sharp $\ell_2$ bound for the EM fixed point $\hat \theta$ in Theorem~\ref{thm:asymptotics}.
The strategy for proving this result consists in an expansion of $\hat \theta$ around the ground truth $\theta^*$.
Recall the definitions of $\bar Q$ and $\hat Q$ in \eqref{eq:qq}.
Since $\hat \theta$ is a fixed point of $\hat Q$, 
we have 
\begin{align*}
\hat\theta - \theta^* &= \hat Q(\hat{\theta}) - \hat Q(\theta^*) + \hat Q(\theta^*) - \theta^* \\
&= \hat \Sigma^\dag \, \frac{d-1}{2N} \sum_{r = 1}^N\bigg(\tanh\Big(\frac{y_rx_r^\top \hat\theta}{\sigma^2}\Big) - \tanh\Big(\frac{y_rx_r^\top \theta^*}{\sigma^2}\Big)\bigg)y_rx_r + \hat Q(\theta^*) - \theta^* .
\end{align*}
Applying Taylor's expansion of the function $\tanh(\cdot)$, we obtain
$$
\hat\theta - \theta^* \approx \hat \Sigma^\dag A (\hat \theta - \theta^*) + \hat Q(\theta^*) - \theta^* ,
$$
where
$A := \frac{d-1}{2N}\sum_{r = 1}^N\frac{y_r^2}{\sigma^2}\tanh' \left(\frac{y_r x_r^\top \theta^*}{\sigma^2} \right) x_r x_r^\top .$
By analyzing the matrix $A$, we then show that $\|\hat\Sigma^\dag A (\hat \theta - \theta^*)\|_2 \ll \|\hat \theta - \theta^*\|_2$ and so
$$
\|\hat \theta - \theta^*\|_2 \approx \|\hat Q(\theta^*) - \theta^*\|_2 .
$$
It remains to study how far the one-step EM iterate $\hat Q(\theta^*)$ moves away from the ground truth $\theta^*$.
Since $\hat Q(\theta^*)$ is an independent sum conditional on the covariates $x_1, \dots, x_N$, and $\theta^*$ is deterministic, we can apply tools from matrix concentration to obtain the desired sharp bound
\begin{equation*}
\|\hat Q(\theta^*) - \theta^*\|_2^2 \le (1+o(1)) \, \sigma^2 \frac{d-1}{2N} \tr(\hat \Sigma^\dag) .
\end{equation*}

\subsection{Infinity operator norm of the inverse covariance}
\label{sec:norm-inv-cov-discus}
As discussed in Section~\ref{sec:contraction-discussion} above, a key ingredient in our analysis of the EM iteration is an upper bound on $\|\hat \Sigma^\dag\|_\infty := \max_{v \in \cH: \|v\|_\infty = 1} \|\hat \Sigma^\dag v\|_\infty$.
Recall that $\E[\hat \Sigma]$ is the identity map on $\cH$ by the discussion after the definition of $\hat \Sigma$ in \eqref{eq:def-samp-cov}.
Hence, it is plausible that $\hat \Sigma$ and $\hat \Sigma^\dag$ have bounded norms.
To prove $\|\hat \Sigma^\dag\|_\infty \le L$ for a quantity $L>0$, we start with the relation (see Lemma~\ref{lem:norm-equiv})
$$
\|\hat \Sigma^\dag \|_\infty = \Big( \min_{u \in \cH : \|u\|_\infty = 1} \|\hat \Sigma u\|_\infty \Big)^{-1} .
$$
It therefore suffices to prove
\begin{equation}
\min_{u \in \cH : \|u\|_\infty = 1} \|\hat \Sigma u\|_\infty \ge 1/L .
\label{eq:connectivity-lower-bound-discussion}
\end{equation}

The proof of \eqref{eq:connectivity-lower-bound-discussion} is the most technical step in our analysis and spans Lemmas~\ref{lem:u-infty-bd}, \ref{lem:sigma-infty-bd}, \ref{lem:fixed-sets-bd}, \ref{lem:two-level-bd}, and Proposition~\ref{prop:sigma-dag-bound}.
In short, it is not sufficient to control $\|\hat \Sigma u\|_\infty$ for each fixed $u$ and then apply a union bound.
Instead, we carefully split the constraint set $\{u \in \cH : \|u\|_\infty = 1\}$ into a union of subsets according to the sizes of entries $u_1, \dots, u_d$.
On each subset consisting of vectors $u$, we use the specific properties of sizes of entries $u_1, \dots, u_d$ to bound $\|\hat \Sigma u\|_\infty$ simultaneously for all $u$ in that subset.
Finally, we combine all the subsets using a union bound.
See Section~\ref{sec:sample-covariance-proof} for details.

While \eqref{eq:connectivity-lower-bound-discussion} arises as a technical ingredient in our work, we believe the analysis of the quantity $\min_{u \in \cH : \|u\|_\infty = 1} \|\hat \Sigma u\|_\infty$ is interesting in its own right in view of its connection to the literature.
First, consider the graph $G$ with vertex set $[d]$ and edge set $\{(i_r, j_r) : x_r = e_{i_r} - e_{j_r}, \, r \in [N]\}$.
Then the graph $G$ resembles an Erd\H{o}s--R\'enyi graph $G\big( d, N/\binom{d}{2} \big)$ except that the edges are sampled with replacement.
Moreover, it is not hard to see that $\hat \Sigma$ is a rescaled version of the Laplacian matrix of the graph $G$, which is a well-researched object.
In fact, if the $\ell_\infty$-norms in \eqref{eq:connectivity-lower-bound-discussion} were replaced by the $\ell_2$-norms, then the quantity $\min_{u \in \cH : \|u\|_2 = 1} \|\hat \Sigma u\|_2$ would be the second-smallest eigenvalue of the Laplacian matrix, known as the algebraic connectivity or the Fiedler value.
This eigenvalue for an Erd\H{o}s--R\'enyi graph has long been studied; see, e.g., \cite{feige2005spectral,coja2007laplacian,chung2011spectra,kolokolnikov2014algebraic}.
However, we are not aware of an existing bound for the $\ell_\infty$-norm that would imply \eqref{eq:connectivity-lower-bound-discussion}.

Moreover, if the minimization problem in \eqref{eq:connectivity-lower-bound-discussion} were over $u \in \{-1,1\}^d$, then it would be related to the vast literature on discrepancy theory: $\min_{u \in \{-1,1\}^d} \|\hat \Sigma u\|_\infty$ is the combinatorial discrepancy of the matrix $\hat \Sigma$.
We refer the interested reader to the celebrated work by \cite{spencer1985six} which has generated extensive research on this topic, and see, e.g., \cite{perkins2021frozen,abbe2022proof,altschuler2022fluctuations} for recent breakthroughs. The relaxation that we study is thus a natural object.

\subsection{Proof organization}
Finally, we overview how the proofs are organized in Section~\ref{sec:analysis}.
Recall that our results do not require a fresh sample for each step of the EM iteration.
To achieve this, we first condition on a high-probability event globally, and then show that the EM operator is contractive deterministically on this event.
More precisely, this high-probability event consists of two sets of conditions: Section~\ref{sec:prelim} proves that certain preliminary conditions hold for the observations $(x_r, y_r)_{r=1}^N$ with high probability (Lemma~\ref{lem:high-prob-bds});
Section~\ref{sec:sample-covariance-proof} mainly establishes a high probability bound on $\|\hat \Sigma^\dag\|_\infty$ (Proposition~\ref{prop:sigma-dag-bound}) together with a few other conditions on $\hat \Sigma$.
With these conditions assumed, Lemma~\ref{lem:easy-contraction} in Section~\ref{sec:convergence} shows that the EM operator is contractive locally around $\theta^*$, and note that this lemma is completely deterministic.
Finally, using an additional bound on $\|\hat Q(\theta^*) - \theta^*\|_\infty$ in Section~\ref{sec:one-step-iter}, Proposition~\ref{prop:em-contraction} in Section~\ref{sec:convergence} summarizes the convergence of the EM sequence and directly leads to Theorem~\ref{thm:em-contraction}.

For Theorem~\ref{thm:asymptotics}, we first need a sharp analysis of $\|\hat Q(\theta^*) - \theta^*\|_2$ in Section~\ref{sec:one-step-iter}: its expectation and concentration are controlled by Lemmas~\ref{lem:l2-err-exp} and~\ref{lem:theta_ell2_bd} respectively.
Then, in Section~\ref{sec:low-noise}, we analyze the lower order terms and conclude with Proposition~\ref{prop:asymptotics}, which immediately implies Theorem~\ref{thm:asymptotics}.


\section{Analysis of the EM algorithm}
\label{sec:analysis}

Throughout the proof section, we use $C, C_1, C'$, etc., to denote universal, positive constants that may change from line to line.
Recall the definitions of $\bar Q$ and $\hat Q$ in \eqref{eq:qq} and that $y_r = z_r \, x_r^\top \theta^* + \epsilon_r$ where $z_r = \pm 1$. 
Since $\tanh(\cdot)$ is an odd function, we may assume without loss of generality that $z_r = 1$ for all $r = 1, \dots, N$, which does not change the EM operator. 
Therefore, we take the liberty of assuming that the observations take the form $y_r = x_r^\top\theta^* + \eps_r$ when analyzing the EM operator.

\subsection{Preliminaries}
\label{sec:prelim}

For $i \in [d]$, define $\cR_i$ to be the set of indices $r \in [N]$ for which the entry $(x_r)_i$ is non-zero, i.e.,
\begin{equation}
\cR_i := \big\{ r \in [N] : x_r = \pm (e_i - e_j) , \, j \in [d] \setminus \{i\} \big\} ,
\label{eq:def-ri}
\end{equation}
where $\pm$ is $+$ if $i<j$ and is $-$ if $i>j$. 
In addition, for a fixed separation parameter $\Delta > 0$, define
\begin{subequations}
\begin{align}
S_i(\Delta) &:= \{j \in [d] \setminus \{i\} : |\theta_i^* - \theta_j^*| \leq \Delta\}  \quad \text{ and }
\label{def:S_i} \\
\cR_i(\Delta) &:= \{r \in [N] : x_r = \pm(e_i - e_j), \, j \in S_i(\Delta) \}. \label{def:R_i}
\end{align}
\end{subequations}
In other words, $S_i(\Delta)$ is the set of indices $j$ such that $\theta^*_j$ is close to $\theta^*_i$, and $\cR_i(\Delta)$ is a subset of $\cR_i$ consisting of indices $r \in [N]$ for which $x_r$ compares $\theta^*_i$ to a nearby parameter $\theta^*_j$.

\begin{lemma}
\label{lem:high-prob-bds}
There is an absolute constant $C > 0$ such that the following holds for any $\delta \in (0,0.1)$.
Suppose that $N \ge C d \log \frac {d}{\delta} .$
It holds with probability at least $1 - \delta$ that, for all $i \in [d]$ and all $\Delta \ge \sigma$:
\begin{itemize}
\item
$\frac{1.9 N}{d} \le |\cR_i| \le \frac{2.1 N}{d}$;

\item
$\sum_{r \in \cR_i(\Delta)} |y_r| \le C \Delta \left( |S_i(\Delta)| \frac{N}{d^2} + \log\frac {d}{\delta} \right)$;

\item
$\sum_{r \in \cR_i(\Delta)} y_r^2 \le C \Delta^2 \left( |S_i(\Delta)| \frac{N}{d^2} + \log\frac {d}{\delta} \right)$.
\end{itemize}
\end{lemma}

\begin{proof}
For a fixed $i \in [d]$, each $x_r$ is equal to $\pm(e_i - e_j)$ for some $j \in [d] \setminus \{i\}$ with probability $(d-1)/\binom{d}{2} = 2/d$.
Hence $|\cR_i| \sim \Bin\left(N, 2/d\right)$.
It then follows from a binomial tail bound for any $\delta \in (0, 0.1)$ that
$$
\bigg| |\cR_i| - \frac{2N}{d} \bigg| \le C_1 \sqrt{\frac Nd \log \frac {d}{\delta} } + C_1 \log \frac {d}{\delta}
$$
with probability at least $1 - \frac \delta d$ for an absolute constant $C_1 > 0$.
Taking a union bound over $i \in [d]$ and using the condition $N \ge C d \log \frac {d}{\delta}$, we obtain the desired bound on $|\cR_i|$.

Similarly, $|\cR_i(\Delta)| \sim \Bin\big(N, \frac{2|S_i(\Delta)|}{d(d-1)}\big)$ and
\[ |\cR_i(\Delta)| \leq \frac{2N|S_i(\Delta)|}{d(d-1)} + C_2 \sqrt{\frac{N|S_i(\Delta)|}{d(d-1)}\log\frac d\delta} + C_2 \log\frac d\delta
\le \frac{3N|S_i(\Delta)|}{d(d-1)} + C_3 \log\frac d\delta \]
with probability at least $1 - \frac{\delta}{d^2}$ for absolute constants $C_2, C_3 > 0$.

Condition on $\cR_i(\Delta)$ henceforth.
For $r \in \cR_i(\Delta)$, we have $|y_r| = |x_r^\top \theta^* + \eps_r| \le \Delta + |\eps_r|$ and so $y_r^2 \le 2 \Delta^2 + 2 \eps_r^2$.
Consequently, by a sub-Gaussian tail bound on $\sum_{r \in \cR_i(\Delta)} |\eps_r|$, we obtain
$$
\sum_{r \in \cR_i(\Delta)} |y_r|
\le \Delta \, |\cR_i(\Delta)| + \sum_{r \in \cR_i(\Delta)} |\eps_r|
\le \Delta \, |\cR_i(\Delta)| + \sqrt{\frac{2}{\pi}} \sigma \, |\cR_i(\Delta)| + C_4 \sigma \sqrt{ |\cR_i(\Delta)| \log \frac d \delta }
$$
with probability at least $1 - \frac{\delta}{d^2}$ for an absolute constant $C_4 > 0$.
Moreover, by a $\chi^2$ tail bound on $\sum_{r \in \cR_i(\Delta)} \eps_r^2$, we obtain
$$
\sum_{r \in \cR_i(\Delta)} y_r^2
\le 2 \Delta^2 \, |\cR_i(\Delta)| + 2 \sum_{r \in \cR_i(\Delta)} \eps_r^2
\le 2 \Delta^2 \, |\cR_i(\Delta)| + 2 \sigma^2 \, |\cR_i(\Delta)| + C_5 \sigma^2 \Big( \sqrt{ |\cR_i(\Delta)| \log \frac d \delta } + \log \frac d \delta \Big)
$$
with probability at least $1 - \frac{\delta}{d^2}$ for an absolute constant $C_5 > 0$.

Combining the above three bounds and using the condition $\Delta \ge \sigma$, we obtain
\begin{align*}
\sum_{r \in \cR_i(\Delta)} |y_r|
&\le 2 \Delta \, \Big( \frac{3N|S_i(\Delta)|}{d(d-1)} + C_3 \log\frac d\delta \Big) + C_4 \sigma \sqrt{ \Big( \frac{3N|S_i(\Delta)|}{d(d-1)} + C_3 \log\frac d\delta \Big) \log \frac d \delta } \\
&\le \frac{7 N \Delta |S_i(\Delta)|}{d(d-1)} + C_6 \Delta \log\frac d\delta
\end{align*}
for an absolute constant $C_6 > 0$
and
\begin{align*}
\sum_{r \in \cR_i(\Delta)} y_r^2
&\le 4 \Delta^2 \, \Big( \frac{3N|S_i(\Delta)|}{d(d-1)} + C_3 \log\frac d\delta \Big) + C_5 \sigma^2 \bigg( \sqrt{ \Big( \frac{3N|S_i(\Delta)|}{d(d-1)} + C_3 \log\frac d\delta \Big) \log \frac d \delta } + \log \frac d \delta \bigg) \\
&\le \frac{13 N \Delta^2 |S_i(\Delta)|}{d(d-1)} + C_7 \Delta^2 \log\frac d\delta
\end{align*}
for an absolute constant $C_7 > 0$.

Finally, note that the sets $S_i(\Delta)$ are nested as $\Delta$ varies, and $S_i(\Delta)$ can take at most $d-1$ values; hence, the same statements are true for $\cR_i(\Delta)$. We can then take a union bound over all ($\le d-1$) possibilities for $\cR_i(\Delta)$ as $\Delta$ varies, together with a union bound over $i \in [d]$ to conclude. 
\end{proof}

\subsection{Sample covariance}
\label{sec:sample-covariance-proof}

In this section, we study the sample covariance matrix \eqref{eq:def-samp-cov}.
Recall that $\cH$ denotes the orthogonal complement of $\bone$ in $\R^d$, and the pseudoinverse $\hat \Sigma^\dag$ is the inverse of $\hat \Sigma$ when viewed as a map on $\cH$.
We first bound the spectral norms of $\hat \Sigma$ and $\hat \Sigma^\dag$. 

\begin{proposition}\label{prop:op_norm}
There is an absolute constant $C > 0$ such that the following holds for any $\delta \in (0,0.1)$.
If $N \ge C d \log \frac {d}{\delta}$, then
with probability at least $1 - \delta$, 
\begin{itemize}
\item
$\|\hat\Sigma\|_{op} \leq 3$;

\item
$\hat \Sigma$ has $d-1$ nonzero eigenvalues;

\item
$\|\hat\Sigma^\dag\|_{op} \leq 5$.
\end{itemize}
\end{proposition}

\begin{proof}
Since $\hat \Sigma$ is positive semidefinite, we have
\(\|\hat \Sigma\|_{op} 
= \lambda_{\max}(\hat \Sigma).\)
Let us first upper-bound $\lambda_{\max}(\hat\Sigma)$.
Let $v \in \R^d$ such that $\|v\|_2 = 1$. By the definition \eqref{eq:def-samp-cov}, we have
\begin{align*}
v^\top\hat\Sigma v &= \frac{d-1}{2N}\sum_{r = 1}^N(x_r^\top v)^2 
= \frac{d-1}{2N}\sum_{r = 1}^N(v_{i_r} - v_{j_r})^2 
\leq \frac{d-1}{N}\sum_{r = 1}^N(v_{i_r}^2 + v_{j_r}^2) 
= \frac{d-1}{N}\sum_{i = 1}^d|\mathcal{R}_i|v_i^2 ,
\end{align*}
where the inequality follows since $(a-b)^2 \leq 2(a^2+b^2)$ for any $a,b\in \R$, and $\cR_i$ is defined in \eqref{eq:def-ri}. 
By the bound on $|\mathcal{R}_i|$ from Lemma \ref{lem:high-prob-bds}, we conclude that with probability at least $1 - \delta$, 
$$
\|\hat \Sigma\|_{op} = \max_{v \in \R^d : \|v\|_2 = 1} v^\top \hat \Sigma v \le \frac{d-1}{N} \cdot \frac{2.1N}{d} \le 3.
$$

Note since $\hat\Sigma^\dag$ is also positive semidefinite, we have $\|\hat\Sigma^\dag\|_{op} = \lambda_{\max}(\hat\Sigma^\dag)$. Furthermore, we have
\[\lambda_{\max}(\hat\Sigma^\dag) = 1/\lambda,\]
where $\lambda$ is the minimum nonzero eigenvalue of $\hat\Sigma$. It suffices to lower-bound the smallest eigenvalue of $\hat \Sigma$ as a map on $\cH$, and then this eigenvalue is precisely $\lambda$.

To this end, let us make a few definitions.
First define the set
\begin{align*}
S & :=  \{v\in \cH\,:\, \|v\|_2 = 1\}.
\end{align*} 
Then define the random matrices
\begin{align*}
D &\in \R^{d\times d}, \text{ such that } D_{ij} = \left\{\begin{array}{cc}
|\mathcal{R}_i| & i = j, \\
0 & i \neq j,
\end{array}\right. \\
A &\in \R^{d\times d}, \text{ such that } A_{ij} = \left\{\begin{array}{cc}
\sum_{r = 1}^N\bbone\{x_r = \pm(e_i - e_j)\} & i \neq j, \\
0 & i = j,
\end{array}\right. \\
M &= \frac{N}{\binom{d}{2}}(J - I) - A,
\end{align*}
where $J = \bone\bone^\top$.
With these in hand, we have the following
\begin{equation}
\lambda = \min_{v\in S}v^\top\hat\Sigma v = \frac{d-1}{2N}\min_{v\in S}v^\top(D - A)v \geq \frac{d-1}{2N}\left(\min_{v\in S}v^\top Dv - \max_{u \in S}u^\top Au\right).
\label{eq:lambda-lwbd}
\end{equation}

Let us consider the first term. We have
\begin{align}\label{eq:degree_diag_lb}
\min_{v\in S} v^\top Dv = \min_{v\in S} \sum_{i = 1}^d|\mathcal{R}_i|v_i^2 \geq \min_{v\in S} \sum_{i = 1}^d \frac{1.9 N}{d}\,v_i^2 = \frac{1.9 N}{d},
\end{align}
where the inequality holds with probability at least $1-\delta$ by Lemma \ref{lem:high-prob-bds}. 

Let us now consider the second term. 
Note that $(I-J) u = u$ for $u \in S$. 
By the definition of $A$, we have
\begin{align}\label{eq:adjacency_decomposition}
\max_{u \in S}u^\top Au 
&= \max_{u \in S} u^\top (-M)u + \frac{N}{\binom{d}{2}} u^\top (I-J) u  \nonumber \\
&\leq \max_{u, v\in S}u^\top(-M)v + \frac{N}{\binom{d}{2}} \nonumber \\
&\leq \|(-M)\|_{op} + \frac{N}{\binom{d}{2}} = \|M\|_{op} + \frac{N}{\binom{d}{2}}.
\end{align}
It remains to bound $\|M\|_{op}$ using Lemma~\ref{lem:matrix_bernstein}. 
To this end, let us define random matrices
$$
X_r = \frac{1}{\binom{d}{2}}(J - I) - \sum_{1\leq i < j \leq d}\bbone\{x_r = \pm(e_i - e_j)\}(e_ie_j^\top + e_je_i^\top), 
$$
and then $M = \sum_{r = 1}^N X_r$.
First, let us show $\EE[X_r] = 0$. We have
\begin{align*}
\EE[X_r] = \frac{1}{\binom{d}{2}}(J - I) - \sum_{1\leq i < j \leq d}\frac{1}{\binom{d}{2}}(e_ie_j^\top + e_je_i^\top) 
=  \frac{1}{\binom{d}{2}}\bigg(J - I - \sum_{i \neq j}e_ie_j^\top\bigg) = 0.
\end{align*}
Next, we find a bound on $\|X_r\|_{op}$. Note that $\sum_{1\leq i < j \leq d}\bbone\{x_r = \pm(e_i - e_j)\} = 1$ for every $r\in [N]$. With this in mind, we have
\[\|X_r\|_{op} \leq \frac{1}{\binom{d}{2}}(\|J\|_{op} + \|I\|_{op}) + \sum_{1\leq i < j \leq d}\bbone\{x_r = \pm(e_i - e_j)\} \, \|e_ie_j^\top + e_je_i^\top\|_{op} = \frac{d+1}{\binom{d}{2}} + 1 \leq 2.\]
Finally, we need a bound on $\|\sum_{r = 1}^N\EE[X_r^2]\|_{op}$. To this end, we compute
\begin{align*}
X_r^2 &= \frac{1}{\binom{d}{2}^2}((d-2)J + I) + \sum_{1 \leq i < j \leq d}\bbone\{x_r = \pm (e_i - e_j)\}(e_ie_i^\top + e_je_j^\top) \\
&\quad - \frac{1}{\binom{d}{2}}\sum_{1 \leq i < j \leq d}\bbone\{x_r = \pm (e_i - e_j)\}\left(\bone e_j^\top + \bone e_i^\top + e_i\bone^\top + e_j \bone ^\top - 2(e_ie_j^\top + e_je_i^\top)\right),
\end{align*}
which implies
\begin{align*}
\EE[X_r^2] &= \frac{1}{\binom{d}{2}^2}((d-2)J + I) + \frac{1}{\binom{d}{2}}\sum_{1 \leq i < j \leq d}(e_ie_i^\top + e_je_j^\top) \\
&\quad - \frac{1}{\binom{d}{2}^2}\sum_{1 \leq i < j \leq d}\left(\bone e_j^\top + \bone e_i^\top + e_i\bone^\top + e_j \bone ^\top - 2(e_ie_j^\top + e_je_i^\top)\right) \\
&= \frac{1}{\binom{d}{2}^2}((d-2)J + I) + \frac{d-1}{\binom{d}{2}}\,I - \frac{2}{\binom{d}{2}^2}((d-2)J + I) \\
&= \frac{d-1}{\binom{d}{2}}\,I - \frac{1}{\binom{d}{2}^2}((d-2)J + I).
\end{align*}
Finally, we have
\[v(M) = \left\| N \E[X_1^2] \right\|_{op} \leq N\left(\frac{2}{d} + \frac{4(d-1)^2}{d^2(d-1)^2}\right) \leq \frac{4N}{d}.\]
By Lemma~\ref{lem:matrix_bernstein}, we have that for some absolute constant $C' > 0$,
\[\|M\|_{op} \leq C' \left(\sqrt{\frac{N}{d}\log \frac{d}{\delta}} + \log \frac{d}{\delta}\right),\]
with probability at least $1 - \delta$. 

Putting the above together with \eqref{eq:degree_diag_lb} and \eqref{eq:adjacency_decomposition}, we have
\[\min_{v\in S}v^\top Dv - \max_{u \in S}u^\top Au \geq \frac{1.9 N}{d} - \frac{N}{\binom{d}{2}} - C'\left(\sqrt{\frac{N}{d}\log \frac{d}{\delta}} + \log \frac{d}{\delta}\right),\]
with probability at least $1 - 2\delta$. Since $d \ge 3$, combining this with \eqref{eq:lambda-lwbd} yields that for some absolute constant $C'' > 0$,
\[\lambda \geq 0.3 - C''\left(\sqrt{\frac{d}{N}\log \frac{d}{\delta}} + \frac{d}{N}\log \frac{d}{\delta}\right),\]
with probability at least $1- 2\delta$. In particular, for $N \geq C\,d\log \frac{d}{\delta}$, we have $\lambda \ge 0.2$ and 
\[ \|\hat\Sigma^\dag\|_{op} = 1/\lambda \leq 5 \]
with probability at least $1-2\delta$.
\end{proof}

\begin{lemma}
\label{lem:trace-inv-lower}
There is an absolute constant $C > 0$ such that the following holds for any $\delta \in (0,0.1)$.
Suppose that $N \ge C d \log \frac d\delta .$
With probability at least $1 - \delta$, we have
\[ \tr(\hat\Sigma^\dag) \ge (d-1)/3 .\]
\end{lemma}

\begin{proof}
Condition on the event where the bounds in Proposition~\ref{prop:op_norm} hold.
Let $\lambda_1 \ge \cdots \ge \lambda_{d-1} > 0$ be the nonzero eigenvalues of $\hat \Sigma$.
Then we have $\lambda_1 \le 3$ and so
\(
\tr(\hat\Sigma^\dag) = \sum_{i=1}^{d-1} \frac{1}{\lambda_i} \ge \frac{d-1}{3} .
\)
\end{proof}

For any linear map $A : \cH \to \cH$, define the norm
\begin{equation*}
\|A\|_\infty := \max_{v \in \cH: \|v\|_\infty = 1} \|A v\|_\infty .
\end{equation*}
We bound the norm $\|\hat \Sigma^\dag\|_\infty$ in the rest of this section: starting with Lemmas~\ref{lem:u-infty-bd}, \ref{lem:sigma-infty-bd}, \ref{lem:fixed-sets-bd}, and \ref{lem:two-level-bd}, the bound is completed in Proposition~\ref{prop:sigma-dag-bound}.

\begin{lemma}
\label{lem:norm-equiv}
Suppose that
$$
\min_{u \in \cH : \|u\|_\infty = 1} \|\hat \Sigma u\|_\infty > 0.
$$
Then we have
$$
\|\hat \Sigma^\dag \|_\infty = \left( \min_{u \in \cH : \|u\|_\infty = 1} \|\hat \Sigma u\|_\infty \right)^{-1} .
$$
\end{lemma}

\begin{proof}
By the assumption, there is no vector $u \in \cH$ such that $\hat \Sigma u = 0$, so $\hat \Sigma$ is invertible on $\cH$.
Then we have
$$
\| \hat \Sigma^\dag  \|_\infty
= \max_{v \in \cH : \|v\|_\infty = 1} \|\hat \Sigma^\dag v\|_\infty
= \max_{u \in \cH : \|\hat \Sigma u\|_\infty = 1} \|u\|_\infty
= \max_{u \in \cH : \|\hat \Sigma u\|_\infty > 0} \frac{\|u\|_\infty}{\|\hat \Sigma u\|_\infty}
= \max_{u \in \cH : \|u\|_\infty = 1} \frac{1}{\|\hat \Sigma u\|_\infty}
$$
which yields the desired result.
\end{proof}

\begin{comment}
Let us define
\[s :=  \|\hat \Sigma^\dag\|_\infty, \quad t  :=  \min_{\substack{\|u\|_\infty = 1, \\ u\perp \mathbf{1}}}\|\hat \Sigma u\|_\infty.\]
It follows from our earlier observation that $t > 0$.
We claim that $s = \frac{1}{t}$.
Since $t > 0$, $\frac{1}{t}$ is well defined.
Consider the following
\begin{align*}
s &= \max_{\|u\|_\infty = 1} \|\hat \Sigma^\dag u\|_\infty \\
&= \max_{\|\hat \Sigma v\|_\infty = 1} \|v\|_\infty \\
&= \max_{\|\hat \Sigma v\|_\infty > 0} \frac{\|v\|_\infty}{\|\hat \Sigma v\|_\infty} \\
&= \max_{\substack{\|v\|_\infty = 1, \\ \|\hat \Sigma v\|_\infty > 0}} \frac{1}{\|\hat \Sigma v\|_\infty} \\
&= \frac{1}{\min\limits_{\substack{\|v\|_\infty = 1, \\ v\perp \mathbf{1}}}\|\hat \Sigma v\|_\infty} \\
&= \frac{1}{t}.
\end{align*}
It remains to provide a lower bound on $t$.
\end{comment}

\begin{lemma}\label{lem:u-infty-bd}
There exists an absolute constant $C > 0$ such that the following holds for any fixed $\delta \in (0,0.1)$, $\kappa \in (0,1]$, and $u \in \cH$ with $\|u\|_\infty = 1$.
Define $I := \{i \in [d] \,:\, u_i \geq \kappa\}.$
Suppose that $N |I| \ge d \log \frac 1\delta$.
With probability at least $1 - \delta$, we have
$$
\|\hat \Sigma u\|_\infty \ge \kappa - C \sqrt{ \frac{d}{N |I|} \log \frac{1}{\delta} } .
$$
\end{lemma}

\begin{proof}
For any $i \in [d]$, we have
$$
(\hat \Sigma u)_i
= \sum_{j = 1}^d \hat \Sigma_{ij} u_j
= \sum_{j = 1}^d \hat \Sigma_{ij} (u_j - u_i) ,
$$
since $\sum_{j=1}^d \hat \Sigma_{ij} = (\hat \Sigma \bone)_i = 0$.
Then, by the definition of $\hat \Sigma$,
\begin{equation}
(\hat \Sigma u)_i
=  \frac{d-1}{2N} \sum_{j \in [d] \setminus \{i\}} \sum_{r = 1}^N \big(x_r x_r^\top\big)_{ij} (u_j - u_i)
= \frac{d-1}{2N} \sum_{r = 1}^N \sum_{j \in [d] \setminus \{i\}}(u_i - u_j) \bbone\{x_r = \pm (e_i - e_j)\} .
\label{eq:sigma-u-i}
\end{equation}
Then it holds that
\begin{equation}
\|\hat \Sigma u\|_\infty
\geq \frac{1}{|I|}\sum_{i \in I}(\hat \Sigma u)_i
= \frac{d-1}{2N |I|} \sum_{r = 1}^N \sum_{i \in I} \sum_{j \in [d] \setminus \{i\}}(u_i - u_j) \bbone\{x_r = \pm (e_i - e_j)\} .
\label{eq:sigma-u-lower}
\end{equation}

For any $r \in [N]$, define
$$
X_r := \frac{d-1}{2} \sum_{i \in I} \sum_{j \in [d] \setminus \{i\}}(u_i - u_j) \bbone\{x_r = \pm (e_i - e_j)\} .
$$
Then we have
$$
\E[X_r] = \frac{d-1}{2} \sum_{i \in I} \sum_{j \in [d] \setminus \{i\}} (u_i - u_j) \frac{2}{d(d-1)}
= \sum_{i \in I} \frac{d u_i - \sum_{j=1}^d u_j}{d}
= \sum_{i \in I} u_i \in \big[ \kappa \, |I| , |I| \big]
$$
since $\sum_{j=1}^d u_j = 0$ and $\kappa \le u_i \le \|u\|_\infty = 1$ for $i \in I$.
Moreover, it follows that
$$
|X_r - \E[X_r]| \le  |X_r| + \E[X_r] \le d \cdot 2 \|u\|_\infty + |I| \le 3 d
$$
and that
$$
\Var(X_r) \le \E[X_r^2]
= \frac{(d-1)^2}{4} \sum_{i \in I} \sum_{j \in [d] \setminus \{i\}} (u_i - u_j)^2 \frac{2}{d(d-1)}
\le 2 d \, |I|.
$$
Since $X_1, \dots, X_N$ are independent, Bernstein's inequality together with \eqref{eq:sigma-u-lower} implies that
$$
\|\hat \Sigma u\|_\infty \ge \frac{1}{N|I|} \sum_{r=1}^N X_r \ge \kappa - C_1 \bigg( \sqrt{ \frac{d}{N |I|} \log \frac{1}{\delta} } + \frac{d}{N |I|} \log \frac{1}{\delta} \bigg)
$$
with probability at least $1 - \delta$ for an absolute constant $C_1>0$.
This completes the proof since $N |I| \ge d \log \frac 1\delta$ by assumption.
\end{proof}

\begin{lemma}
\label{lem:sigma-infty-bd}
There exists an absolute constant $C > 0$ such that the following holds for any fixed $\delta \in (0,0.1)$ and $\alpha, \kappa \in (0,1]$.
Suppose that $N \ge C \frac{d}{\alpha \kappa^2} \log \frac{2}{\kappa \delta}$.
Define
$$
\cU := \big\{ u \in \cH : \|u\|_\infty = 1, \, |\{i \in [d] \,:\, u_i \geq \kappa\}| \ge \alpha d \big\}.
$$
With probability at least $1 - \delta$, we have
$$
\min_{u \in \cU} \|\hat \Sigma u\|_\infty \ge \kappa/2 .
$$
\end{lemma}

\begin{proof}
For $\eps \in (0, 1)$, the set
$$
\cN := \{u \in \cU : u_i = 0, \pm \eps, \pm 2 \eps, \dots, \pm \lfloor 1/\eps \rfloor \eps \text{ for all } i \in [d] \}
$$
is clearly an $\eps$-net of $\cU$ in the $\ell_\infty$ norm, and it has cardinality $|\cN| \le (2/\eps)^d$.

For any $u  \in \cU$, let $v \in \cN$ be such that $\|u - v\|_\infty \le \eps$.
Then we have
\begin{equation}
\|\hat \Sigma v\|_\infty = \|\hat \Sigma(v-u + u)\|_\infty \leq \|\hat \Sigma\|_\infty\|v-u\|_\infty + \|\hat \Sigma u\|_\infty \leq \eps\|\hat \Sigma\|_\infty + \|\hat \Sigma u\|_\infty .
\label{eq:net-arg}
\end{equation}
By the definition of $\hat \Sigma$ and the fact that $\sum_{j=1}^d \hat \Sigma_{ij} = 0$ for any $i \in [d]$, we obtain
\begin{align*}
\|\hat \Sigma\|_\infty
= \max_{i \in [d]}\sum_{j = 1}^d|\hat{\Sigma}_{ij}|
= 2 \max_{i \in [d]} \sum_{j \in [d] \setminus \{i\}} |\hat{\Sigma}_{ij}|
&= \frac{d-1}{N} \max_{i \in [d]} \sum_{j \in [d] \setminus \{i\}} \bigg| \sum_{r=1}^N (x_r x_r^\top)_{ij} \bigg| \\
&= \frac{d-1}{N} \max_{i \in [d]} \sum_{j \in [d] \setminus \{i\}} \sum_{r=1}^N \bbone\{x_r = \pm(e_i - e_j)\} .
\end{align*}
It then follows from the definition of $\cR_i$ in \eqref{eq:def-ri} and Lemma \ref{lem:high-prob-bds} that
$$
\|\hat \Sigma\|_\infty
= \frac{d-1}{N} \max_{i \in [d]} |\mathcal{R}_i| \leq 3
$$
with probability at least $1 - \delta/2$.
Combining this bound with \eqref{eq:net-arg} gives
$$
\min_{u \in \cU} \|\hat \Sigma u\|_\infty
\ge \min_{v \in \cN} \|\hat \Sigma v\|_\infty - 3 \eps .
$$

Recall that $|\{i \in [d] \,:\, v_i \geq \kappa\}| \ge \alpha d$ for every $v \in \cN \subset \cU$.
Let $\eps = \kappa/10$.
We apply Lemma~\ref{lem:u-infty-bd} with $\delta$ replaced by $\frac{\delta}{2(2/\eps)^d}$ and a union bound over all $v \in \cN$ to obtain that, with probability at least $1 - \delta/2$,
$$
\min_{v \in \cN} \|\hat \Sigma v\|_\infty \ge \kappa - C_1 \sqrt{ \frac{d}{N \alpha d} \log \frac{2(2/\eps)^d}{\delta} }
\ge \kappa - C_2 \sqrt{ \frac{d}{N \alpha} \log \frac{2}{\kappa \delta} }
$$
for absolute constants $C_1, C_2 > 0$.

In view of the above two displays together with the condition $N \ge C \frac{d}{\alpha \kappa^2} \log \frac{2}{\kappa \delta}$ for a large constant $C > 0$, we conclude that $\min_{u \in \cU} \|\hat \Sigma u\|_\infty \ge \kappa/2$ with probability at least $1 - \delta$.
\end{proof}

\begin{lemma}
\label{lem:fixed-sets-bd}
There exists an absolute constant $C > 0$ such that the following holds for any fixed $\delta \in (0,0.1)$, $0 \le \lambda < \kappa \le 1$, and nonempty subsets $I \subset J \subset [d]$.
With probability at least $1 - \delta$, we have that
$$
\|\hat \Sigma u\|_\infty
\ge (\kappa - \lambda) \frac{d-|J|}{d} - \frac{|I|}{d} - C \bigg( \sqrt{ \frac{d}{N|I|} \log \frac{1}{\delta} } + \frac{d}{N |I|} \log \frac{1}{\delta} \bigg)
$$
for all $u \in \cH$ with $\|u\|_\infty = 1$, $\{i \in [d] : u_i \ge \kappa\} = I$, and $\{i \in [d] : u_i \ge \lambda\} = J$.
\end{lemma}

\begin{proof}
By \eqref{eq:sigma-u-i}, we have
\begin{align*}
\|\hat \Sigma u\|_\infty
\geq \frac{1}{|I|}\sum_{i \in I} (\hat \Sigma u)_i
&=  \frac{d-1}{2N |I|}  \sum_{r = 1}^N \sum_{i \in I} \sum_{j \in [d] \setminus \{i\}}(u_i - u_j) \bbone\{x_r = \pm (e_i - e_j)\} \notag \\
&= \frac{d-1}{2N |I|}  \sum_{r = 1}^N \bigg( \sum_{i \in I} \sum_{j \in I \setminus \{i\}} (u_i - u_j) \bbone\{x_r = \pm (e_i - e_j)\} \notag \\
& \qquad \qquad \qquad + \sum_{i \in I} \sum_{j \in J \setminus I} (u_i - u_j) \bbone\{x_r = \pm (e_i - e_j)\} \notag \\
& \qquad \qquad \qquad + \sum_{i \in I} \sum_{j \in [d] \setminus J} (u_i - u_j) \bbone\{x_r = \pm (e_i - e_j)\} \bigg) . \notag
\end{align*}
In view of the assumptions $\|u\|_\infty = 1$, $\{i \in [d] : u_i \ge \kappa\} = I$, and $\{i \in [d] : u_i \ge \lambda\} = J$, we see that for $i \in I$,
$$
u_i - u_j \ge
\begin{cases}
-1 & \text{ if } j \in I \setminus \{i\} ,\\
0 & \text{ if } j \in J \setminus I ,\\
\kappa - \lambda & \text{ if } j \in [d] \setminus J.
\end{cases}
$$
It then follows that
\begin{align}
\|\hat \Sigma u\|_\infty
&\ge - \frac{1}{N |I|} \sum_{r = 1}^N \frac{d-1}{2}  \sum_{i \in I} \sum_{j \in I \setminus \{i\}} \bbone\{x_r = \pm (e_i - e_j)\} \label{eq:sum-rij-1} \\
& \quad + \frac{\kappa - \lambda}{N |I|} \sum_{r = 1}^N \frac{d-1}{2}  \sum_{i \in I} \sum_{j \in [d] \setminus J}  \bbone\{x_r = \pm (e_i - e_j)\}. \label{eq:sum-rij-2}
\end{align}

We now bound the above two terms.
To control \eqref{eq:sum-rij-1}, we define
$$
X_r := \frac{d-1}{2} \sum_{i \in I} \sum_{j \in I \setminus \{i\}} \bbone\{x_r = \pm (e_i - e_j)\}
$$
for $r \in [N]$.
Then we have
$$
\E[X_r] = \frac{d-1}{2} \sum_{i \in I} \sum_{j \in I \setminus \{i\}} \frac{2}{d(d-1)}
= \frac{|I| (|I| - 1)}{d}
\le \frac{|I|^2}{d} , \qquad
|X_r - \E[X_r]| \le d ,
$$
and
$$
\Var(X_r) \le \E[X_r^2]
= \frac{(d-1)^2}{4} \sum_{i \in I} \sum_{j \in I \setminus \{i\}} \frac{2}{d(d-1)}
\le \frac{|I|^2}{2}.
$$
Since $X_1, \dots, X_N$ are independent, Bernstein's inequality implies that
$$
\sum_{r=1}^N X_r \le N \frac{|I|^2}{d} + C_1 \bigg( \sqrt{ N |I|^2 \log \frac{1}{\delta} } + d \log \frac{1}{\delta} \bigg)
$$
with probability at least $1 - \delta/2$ for an absolute constant $C_1>0$ and any $\delta \in (0,0.1)$.
Consequently, the term \eqref{eq:sum-rij-1} can be bounded as
$$
\frac{1}{N |I|} \sum_{r = 1}^N \frac{d-1}{2}  \sum_{i \in I} \sum_{j \in I \setminus \{i\}} \bbone\{x_r = \pm (e_i - e_j)\}
\le \frac{|I|}{d} + C_1 \bigg( \sqrt{ \frac 1N \log \frac{1}{\delta} } + \frac{d}{N |I|} \log \frac{1}{\delta} \bigg) .
$$

Analogously, if we define
$$
X'_r := \frac{d-1}{2}  \sum_{i \in I} \sum_{j \in [d] \setminus J}  \bbone\{x_r = \pm (e_i - e_j)\} ,
$$
we can apply Bernstein's inequality to obtain that
$$
\sum_{r=1}^N X'_r \ge N \frac{|I| (d-|J|)}{d} - C_2 \bigg( \sqrt{ N |I| (d-|J|) \log \frac{1}{\delta} } + d \log \frac{1}{\delta} \bigg)
$$
with probability at least $1 - \delta/2$ for an absolute constant $C_2>0$ and any $\delta \in (0,0.1)$.
Then the term \eqref{eq:sum-rij-2} can be bounded as
$$
\frac{\kappa - \lambda}{N |I|} \sum_{r = 1}^N \frac{d-1}{2}  \sum_{i \in I} \sum_{j \in [d] \setminus J}  \bbone\{x_r = \pm (e_i - e_j)\}
\ge (\kappa - \lambda) \frac{d-|J|}{d} - C_2 \bigg( \sqrt{ \frac{d}{N|I|} \log \frac{1}{\delta} } + \frac{d}{N |I|} \log \frac{1}{\delta} \bigg) .
$$

Plugging the above bounds into \eqref{eq:sum-rij-1} and \eqref{eq:sum-rij-2} respectively completes the proof.
\end{proof}

\begin{lemma}
\label{lem:two-level-bd}
There exists an absolute constant $C > 0$ such that the following holds for any fixed $\delta \in (0,0.1)$ and $0 \le \lambda < \kappa \le 1$.
For $u \in \R^d$, define
$$
\ell(u) := |\{i \in [d] : u_i \ge \kappa\}|, \qquad
m(u) := |\{i \in [d] : u_i \ge \lambda\}| .
$$
With probability at least $1 - \delta$, we have that
$$
\|\hat \Sigma u\|_\infty
\ge \frac{\kappa - \lambda}{2}
$$
for all $u \in \cH$ with $\|u\|_\infty = 1$, $m(u) \le \frac d4$, $\ell(u) \le (\kappa - \lambda) \frac d8$, and $N \ge \frac{C d \, m(u)}{(\kappa - \lambda)^2 \ell(u)} \log \frac{d}{\delta}$.
\end{lemma}

\begin{proof}
For fixed positive integers $\ell \le m \le d$, there are at most $\binom{d}{m} \binom{m}{\ell} \le d^m$ pairs of subsets $(I,J)$ such that $|I| = \ell$, $|J| = m$, and $I \subset J \subset [d]$.
Therefore, by Lemma~\ref{lem:fixed-sets-bd} with $\delta$ replaced by $\delta/d^m$ and a union bound over all possible $(I,J)$, we obtain the following: With probability at least $1 - \delta$,
\begin{equation}
\|\hat \Sigma u\|_\infty
\ge (\kappa - \lambda) \frac{d-m}{d} - \frac{\ell}{d} - C_1 \bigg( \sqrt{ \frac{dm}{N \ell} \log \frac{d}{\delta} } + \frac{dm}{N \ell} \log \frac{d}{\delta} \bigg)
\label{eq:sim-ineq}
\end{equation}
for all $u \in \cH$ with $\|u\|_\infty = 1$, $\ell(u) = \ell$, and $m(u) = m$, where $C_1>0$ is an absolute constant, and $\ell(u)$ and $m(u)$ are defined as in the lemma.
Moreover, there are at most $d^2$ possible choices for the pair of integers $(\ell, m)$.
Thus, by a further union bound over $(\ell, m)$, the bound \eqref{eq:sim-ineq} holds for all $u \in \cH$ with $\|u\|_\infty = 1$, provided that the constant $C_1$ is replaced by a larger absolute constant.
The conclusion then follows from the imposed conditions $m \le \frac d4$, $\ell \le (\kappa - \lambda) \frac d8$, and $N \ge \frac{C d \, m}{(\kappa - \lambda)^2 \ell} \log \frac{d}{\delta}$.
\end{proof}

\begin{proposition}
\label{prop:sigma-dag-bound}
There exists an absolute constant $C>0$ such that the following holds for any $\delta \in (0,0.1)$.
Fix an integer $L \ge 2$.
Suppose that
$$
N \ge C \max \Big\{ L^2 d^{\frac{L}{L-1}} \log \frac{dL}{\delta} , L^3 d \log \frac{L}{\delta} \Big\} .
$$
Then it holds with probability at least $1-\delta$ that,
$$
\min_{u \in \cH : \|u\|_\infty = 1} \|\hat \Sigma u\|_\infty \ge \frac{1}{2L}
$$
and, as a result,
$$
\| \hat \Sigma^\dag \|_\infty \le 2 L.
$$
\end{proposition}

\begin{proof}
By Lemma~\ref{lem:norm-equiv}, it suffices to prove the lower bound on $\|\hat \Sigma u\|_\infty$ for $u \in \cH$ with $\|u\|_\infty = 1$.
For $t = 0, 1, \dots, L$, define
$$
\kappa_t := 1 - t/L, \qquad
\ell_t(u) := |\{i \in [d] : u_i \ge \kappa_t\}|.
$$
For $t = 1, 2, \dots, L-1$, define
$$
\cU_t := \bigg\{ u \in \cH : \|u\|_\infty = 1, \, \ell_t(u) \le \frac d4, \, \ell_{t-1}(u) \le \frac{d}{8L}, \, N \ge \frac{C_1 d L^2 \ell_t(u)}{\ell_{t-1}(u)} \log \frac{dL}{\delta} \bigg\} ,
$$
where $C_1 > 0$ is a constant to be chosen.
In addition, define
$$
\cU_L := \bigg\{ u \in \cH : \|u\|_\infty = 1, \, \ell_{L-1}(u) \ge \frac{d}{8L} \bigg\} .
$$
We claim that
\begin{equation}
\bigcup_{t=1}^{L} \cU_t = \bigg\{ u \in \cH : \|u\|_\infty = 1,  \, \max_{i \in [d]} u_i = 1 \bigg\} .
\label{eq:claim-set-union}
\end{equation}

We first finish the proof based on the above claim.
For $u \in \bigcup_{t=1}^{L-1} \cU_t$, we apply Lemma~\ref{lem:two-level-bd} with $\kappa = \kappa_{t-1}$, $\lambda = \kappa_t$, $\ell(u) = \ell_{t-1}(u)$, $m(u) = \ell_t(u)$, and $\delta$ replaced by $\frac{\delta}{2L}$, together with a union bound over $t \in [L-1]$.
Note that $\kappa - \lambda = 1/L$.
Hence we obtain that, with probability at least $1-\delta/2$,
$$
\min_{u \in \bigcup_{t=1}^{L-1} \cU_t} \|\hat \Sigma u\|_\infty \ge \frac{1}{2L} ,
$$
provided that $C_1$ is a sufficiently large absolute constant.
Moreover, for $u \in \cU_L$, we apply Lemma~\ref{lem:sigma-infty-bd} with $\kappa = K_{L-1} = 1/L$ and $\alpha = 1/(8L)$ to obtain that, with probability at least $1-\delta/2$,
$$
\min_{u \in \cU_L} \|\hat \Sigma u\|_\infty \ge \frac{1}{2L} ,
$$
provided that $N \ge C_2 L^3 d \log \frac{L}{\delta}$ for a large constant $C_2 > 0$.

The above two bounds together with \eqref{eq:claim-set-union} imply that, with probability at least $1 - \delta$, we have $\|\hat \Sigma u\|_\infty \ge 1/(2L)$ for all $u \in \cH$ with $\|u\|_\infty = 1$ and $\max_{i \in [d]} u_i = 1$.
Finally, for $u \in \cH$ with $\|u\|_\infty = 1$ and $\min_{i \in [d]} u_i = -1$, it suffices to apply the bound to $-u$.
\end{proof}

\begin{proof}[Proof of Claim~\eqref{eq:claim-set-union}]
Consider $u \in \cH$ with $\|u\|_\infty = 1$ and $\max_{i \in [d]} u_i = 1$.
Suppose that $u \notin \bigcup_{t=1}^{L-1} \cU_t$.
Then for any $t \in [L-1]$, we have $\ell_t(u) > \frac d4$, $\ell_{t-1}(u) > \frac{d}{8L}$, or $N < \frac{C_1 d L^2 \ell_t(u)}{\ell_{t-1}(u)} \log \frac{dL}{\delta}$.
If either of the first two conditions holds, then we have $\ell_t(u) > \frac{d}{8L}$ for some $t \in \{0, 1, \dots, L-1\}$.
Since $\kappa_{L-1} \le \kappa_t$, by the definition of $\ell_t(u)$, we have $\ell_{L-1}(u) \ge \ell_t(u) > \frac{d}{8L}$.
Thus $u \in \cU_L$.
Let us further suppose that $u \notin \cU_L$.
Then, for all $t \in [L-1]$, we must have
$$
N < \frac{C_1 d L^2 \ell_t(u)}{\ell_{t-1}(u)} \log \frac{dL}{\delta}
\qquad \Longleftrightarrow \qquad
\ell_t(u) > \ell_{t-1}(u) \cdot \frac{N}{C_1 L^2 d \log(dL/\delta)} .
$$
Since $\kappa_0 = 1$ and $\ell_0(u) = |\{i \in [d] : u_i = 1\}| \ge 1$, we apply the above bound iteratively to obtain
$$
\ell_{L-1}(u) \ge \bigg(\frac{N}{C_1 L^2 d \log(dL/\delta)}\bigg)^{L-1} \ge d > \frac{d}{8L} ,
$$
where the second inequality holds by the assumption $N \ge C L^2 d^{\frac{L}{L-1}} \log \frac{dL}{\delta}$.
As a result, $u \in \cU_L$, finishing the proof.
\end{proof}

\subsection{One-step iterate from the ground truth}
\label{sec:one-step-iter}

In this subsection, we study the difference between the ground truth $\theta^*$ and the one-step EM iterate $\hat Q(\theta^*)$ from it. 
All the results in this subsection are conditional on a realization of $x_1, \dots, x_N$. 
For brevity, we often take the liberty of using $\E$ and $\p$ to denote the conditional expectation and probability on $x_1, \dots, x_N$ respectively. 
Let us start with a lemma.

\begin{lemma}
Let $X\sim \mathcal{N}(\mu, \sigma^2)$. 
Then we have 
\label{lem:gauss-tanh}
\begin{equation}
\EE \left[ X\tanh\left(\frac{\mu X}{\sigma^2}\right) \right] = \mu ,
\label{eq:gauss-tanh-expect}
\end{equation}
and 
\begin{equation}
\Var \left( X\tanh\left(\frac{\mu X}{\sigma^2}\right) \right) \le  \sigma^2 .
\label{eq:gauss-tanh-var}
\end{equation}
\end{lemma}

\begin{proof}
Equation \eqref{eq:gauss-tanh-expect} is well-known, but we provide a proof for completeness. 
By a change of variable $\nu = \frac{\mu}{\sigma}$ and $Z = \frac{X}{\sigma} \sim \cN(\nu, 1)$, it suffices to show that 
\begin{equation}
\EE \left[ Z \tanh( \nu Z ) \right] = \nu .
\label{eq:tanh-change-of-var}
\end{equation}
We have $\E[Z] = \nu$ and 
\begin{equation}
\EE \left[ Z \left( \tanh( \nu Z ) - 1 \right) \right]  
= \frac{1}{\sqrt{2\pi}} \int_{\R} z \frac{-2 e^{-\nu z}}{e^{\nu z} + e^{-\nu z}} e^{-(z-\nu)^2/2} \, dz 
= - \frac{2 e^{-\nu^2/2}}{\sqrt{2\pi}}  \int_{\R} \frac{z e^{-z^2/2}}{e^{\nu z} + e^{-\nu z}} \, dz 
= 0 ,
\label{eq:tanh-inter}
\end{equation}
because the integrand is an odd function. 
Then \eqref{eq:tanh-change-of-var} immediately follows.

Next, it holds that 
\begin{equation}
\Var \left( X\tanh(\mu X/\sigma^2) \right) 
= \sigma^2 \Var \left( Z \tanh(\nu Z) \right) 
= \sigma^2 \E\left[ (Z \tanh(\nu Z) - \nu)^2 \right] .
\label{eq:tanh-var-1}
\end{equation}
Furthermore, we have 
\begin{align*}
(Z \tanh(\nu Z) - \nu)^2 
&= (Z - \nu)^2 - 2 v Z (\tanh(\nu Z) - 1) + Z^2 (\tanh(\nu Z)^2 - 1) \\
&\le (Z - \nu)^2 - 2 v Z (\tanh(\nu Z) - 1) + Z^2 \tanh(\nu Z) \, (\tanh(\nu Z) - 1)
\end{align*}
since $\tanh(\nu Z) \le 1$. 
Let us compute the expectations of these three terms. 
We have
$$
\E[(Z - \nu)^2] = 1, \qquad 
\E[ 2 v Z (\tanh(\nu Z) - 1) ] = 0 ,
$$
by $Z \sim \cN(\nu,1)$ and \eqref{eq:tanh-inter} respectively. 
For the third term, 
\begin{align*}
\E[Z^2 \tanh(\nu Z) \, (\tanh(\nu Z) - 1)]
&= \frac{1}{\sqrt{2\pi}}
\int_{\R} z^2 \frac{ (e^{\nu z} - e^{-\nu z}) (-2 e^{-\nu z}) }{ (e^{\nu z} + e^{-\nu z})^2 } e^{-(z-\nu)^2/2} \, dz \\
&= \frac{-2 e^{-v^2/2}}{\sqrt{2\pi}}
\int_{\R} z^2 e^{-z^2/2} \frac{ e^{\nu z} - e^{-\nu z} }{ (e^{\nu z} + e^{-\nu z})^2 } \, dz 
= 0 ,
\end{align*}
because  the integrand is an odd function. 
Combining the three terms with \eqref{eq:tanh-var-1} proves \eqref{eq:gauss-tanh-var}.
\end{proof}

We now control $\hat{Q}(\theta^*) - \theta^*$ in the $\ell_\infty$-norm.

\begin{lemma}\label{lem:theta_infty_bd}
Condition on a realization of $x_1, \ldots, x_N$ such that $\| \hat \Sigma^\dag \|_\infty \le 2 L$.
There is an absolute constant $C > 0$ such that for any $\delta \in (0,0.1)$, it holds with (conditional) probability at least $1 - \delta$ that
\[\|\hat{Q}(\theta^*) - \theta^*\|_\infty \leq C \sigma \sqrt{\frac{L\,d}{N}\log\frac d\delta}.\]
\end{lemma}

\begin{proof}
Let us first check that $\E[\hat Q(\theta^*)] = \theta^*$. 
By \eqref{eq:gauss-tanh-expect}, 
we have
\begin{equation}
\EE_{y_r\sim \mathcal{N}(x_r^\top\theta^*, \sigma^2)} \left[ y_r\tanh\left(\frac{y_rx_r^\top\theta^*}{\sigma^2}\right) \right] = x_r^\top\theta^*.
\label{eq:gauss-expect}
\end{equation}
Then, it follows that
\begin{equation}
\E[\bar Q(\theta^*)] = \frac{d-1}{2N} \sum_{r=1}^N x_r x_r^\top \theta^* = \hat \Sigma \, \theta^*, \qquad 
\EE[\hat{Q}(\theta^*)] = \theta^*.
\label{eq:em-expect}
\end{equation}

Next, we study $\hat Q(\theta^*) - \theta^*$, whose $i$th entry is 
\begin{equation}
e_i^\top \big(\hat Q(\theta^*) - \theta^*\big) = e_i^\top \hat \Sigma^\dag \big(\bar Q(\theta^*) - \E[\bar Q(\theta^*)]\big).
\label{eq:row-inn-prod-1}
\end{equation}
For $w \in \cH$, define 
\begin{equation}
f_w(y)  :=  w^\top\bar Q(\theta^*) = \frac{d-1}{2N}\sum_{r = 1}^Ny_rx_r^\top w\tanh\left(\frac{y_rx_r^\top\theta^*}{\sigma^2}\right).
\label{eq:row-inn-prod-2}
\end{equation}
We will show that $f_w(y) - \EE_y[f_w(y)]$ is sub-Gaussian with variance parameter of order $\frac{\sigma^2L\,d}{N}$ when $w^\top$ is a row of $\hat\Sigma^\dag$. 
The proof is similar to that of Proposition~11 of \cite{kwon2019global}, with the main tool being Lemma~\ref{lem:wainwright}, a standard concentration result for Gaussian random variables.
To this end, let $v_r = \frac{y_r - x_r^\top\theta^*}{\sigma} = \frac{\eps_r}{\sigma}$ be i.i.d.\ standard Gaussians and introduce another independent set of i.i.d.\ standard Gaussians $\zeta_1, \dots, \zeta_N$. 
Let us define a new function
\[g_w(v)  :=  f_w(y) = \frac{d-1}{2N}\sum_{r = 1}^N(\sigma v_r + x_r^\top\theta^*)x_r^\top w\tanh\left(\frac{(\sigma v_r + x_r^\top\theta^*)x_r^\top\theta^*}{\sigma^2}\right).\]
For each $r \in [N]$, we have
\begin{align*}
(\nabla g_w(v))_r &= \frac{d-1}{2N}\,x_r^\top w\bigg[\sigma\tanh\left(\frac{(\sigma v_r + x_r^\top\theta^*)x_r^\top\theta^*}{\sigma^2}\right) \\
&\qquad \qquad \qquad \quad + (\sigma v_r + x_r^\top\theta^*)\frac{x_r^\top\theta^*}{\sigma}\tanh'\left(\frac{(\sigma v_r + x_r^\top\theta^*)x_r^\top\theta^*}{\sigma^2}\right)\bigg] \\
&= \frac{d-1}{2N}\,x_r^\top w\sigma \cdot h\left(\frac{(\sigma v_r + x_r^\top\theta^*)x_r^\top\theta^*}{\sigma^2}\right),
\end{align*}
where $h(t)  :=  \tanh(t) + t \cdot \tanh'(t)$. We will use the inequality $|h(t)| \leq 2$ for all $t \in \R$. 
By Lemma~\ref{lem:wainwright}, 
\begin{align*}
&~\EE[\exp(\lambda(f_w(y) - \EE[f_w(y)]))] 
=\EE[\exp(\lambda(g_w(v) - \EE[g_w(v)]))] \\
&\leq \EE_{v,\zeta}\left[\exp\left(\frac{\lambda\pi}{2}\langle\nabla g, \zeta\rangle\right)\right] \\
&= \EE_v\EE_\zeta\left[\exp\left(\frac{\lambda\pi}{2}\,\left(\frac{d-1}{2N}\sum_{r = 1}^N\zeta_rx_r^\top w\sigma \cdot h\left(\frac{(\sigma v_r + x_r^\top\theta^*)x_r^\top\theta^*}{\sigma^2}\right)\right)\right)\right] \\
&= \EE_v\left[\exp\left(\left(\frac{\lambda^2\sigma^2\pi^2(d-1)}{16N}\right)\left(\frac{d-1}{2N}\sum_{r = 1}^N \left( x_r^\top w \cdot h((\sigma v_r + x_r^\top\theta^*)x_r^\top\theta^*) \right)^2 \right)\right)\right] \\
&\leq \exp\left(\left(\frac{\lambda^2\sigma^2\pi^2(d-1)}{4N}\right)\left(\frac{d-1}{2N}\sum_{r = 1}^N(x_r^\top w)^2\right)\right) \\
&\leq \exp\left(\frac{\lambda^2\sigma^2\pi^2d}{4N}\,w^\top\hat\Sigma w\right).
\end{align*}
If $w^\top$ is the $i$th row of $\hat\Sigma^\dag$, we have $w = \hat \Sigma^\dag e_i$ and 
\[ 
w^\top \hat\Sigma w = e_i^\top \hat \Sigma^\dag e_i \le \|\hat \Sigma^\dag\|_\infty \leq 2L 
\]
by assumption.
It follows that
\[\EE[\exp(\lambda(f_w(y) - \EE[f_w(y)]))] \leq \exp\left(\frac{\lambda^2\sigma^2\pi^2 d L}{2N}\right),\]
so $f_w(y)$ is sub-Gaussian with variance parameter $\frac{\sigma^2\pi^2 d L}{N}$. 
Therefore, it holds with probability at least $1 - \delta/d$ that
\[
\big| f_w(y) - \EE[f_w(y)] \big| \leq C \sigma \sqrt{\frac{L d}{N} \log \frac{d}{\delta}}
\]
for a constant $C > 0$. 
In view of \eqref{eq:row-inn-prod-1} and \eqref{eq:row-inn-prod-2}, the left-hand side of the above bound is precisely $|e_i^\top (\hat Q(\theta^*) - \theta^*)|$ when $w^\top$ is the $i$th row of $\hat \Sigma^\dag$. 
Taking a union bound over $i \in [d]$ proves the desired result on $\|\hat Q(\theta^*) - \theta^*\|_\infty$. 
\end{proof}

The following two lemmas provide a sharp control on $\hat{Q}(\theta^*) - \theta^*$ in the $\ell_2$-norm: they respectively bound the expectation and deviation of $\|\hat{Q}(\theta^*) - \theta^*\|_2^2$.

\begin{lemma}
\label{lem:l2-err-exp}
Condition on a realization of $x_1, \ldots, x_N$ such that $\hat \Sigma^\dag$ is the inverse of $\hat \Sigma$ as a map on $\cH$. 
Let $\E[\cdot]$ denote the conditional expectation.
Then we have
$$
\EE \big[ \|\hat{Q}(\theta^*) - \theta^*\|_2^2 \big]
\le \sigma^2 \frac{d-1}{2N} \tr(\hat \Sigma^\dag) .
$$
\end{lemma}

\begin{proof}
By the definition of $\hat Q$ in \eqref{eq:qq-2}, 
it is not difficult to see that
$$
\hat{Q}(\theta^*) - \theta^*
= \bigg( \sum_{r = 1}^Nx_rx_r^\top \bigg)^\dag \sum_{r=1}^N \bigg( \tanh \Big( \frac{ y_r \, x_r^\top \theta^* }{\sigma^2} \Big) y_r - x_r^\top \theta^* \bigg) x_r .
$$
Then it follows that
\begin{align*}
&\E\big[ (\hat{Q}(\theta^*) - \theta^*) (\hat{Q}(\theta^*) - \theta^*)^\top \big] \\
&= \bigg( \sum_{r = 1}^N x_r x_r^\top \bigg)^\dag \Bigg( \sum_{r=1}^N \E \bigg[ \bigg( \tanh \Big( \frac{ y_r \, x_r^\top \theta^* }{\sigma^2} \Big) y_r - x_r^\top \theta^* \bigg)^2 \bigg] x_r x_r^\top \Bigg) \bigg( \sum_{r = 1}^N x_r x_r^\top \bigg)^\dag ,
\end{align*}
where the cross terms vanish thanks to \eqref{eq:gauss-expect}.
By \eqref{eq:gauss-tanh-var}, we have 
$$
\sigma_r := \E \bigg[ \bigg( \tanh \Big( \frac{ y_r \, x_r^\top \theta^* }{\sigma^2} \Big) y_r - x_r^\top \theta^* \bigg)^2 \bigg] \le \sigma^2. 
$$
As a result, the matrix
$$
\sigma^2 \bigg( \sum_{r = 1}^N x_r x_r^\top \bigg)^\dag  - \E\big[ (\hat{Q}(\theta^*) - \theta^*) (\hat{Q}(\theta^*) - \theta^*)^\top \big] 
= \bigg( \sum_{r = 1}^N x_r x_r^\top \bigg)^\dag \Bigg( \sum_{r=1}^N (\sigma^2 - \sigma_r^2) x_r x_r^\top \Bigg) \bigg( \sum_{r = 1}^N x_r x_r^\top \bigg)^\dag
$$
is positive semi-definite. 
In other words, 
$$
\E\big[ (\hat{Q}(\theta^*) - \theta^*) (\hat{Q}(\theta^*) - \theta^*)^\top \big] \preccurlyeq \sigma^2 \bigg( \sum_{r = 1}^N x_r x_r^\top \bigg)^\dag  
$$
in the Loewner order. 
Since the trace operator is monotone in the Loewner order, we obtain 
$$
\EE \big[ \|\hat{Q}(\theta^*) - \theta^*\|_2^2 \big]
= \tr \Big( \E\big[ (\hat{Q}(\theta^*) - \theta^*) (\hat{Q}(\theta^*) - \theta^*)^\top \big] \Big)
\le \sigma^2 \tr \Bigg( \bigg( \sum_{r = 1}^N x_r x_r^\top \bigg)^\dag \Bigg) ,
$$
which completes the proof. 
\end{proof}

\begin{lemma}\label{lem:theta_ell2_bd}
Condition on a realization of $x_1, \dots, x_N$ such that $\|\hat\Sigma\|_{op} \leq 3$ and $\|\hat\Sigma^\dag\|_{op} \leq 5$.
Then there exists an absolute constant $C > 0$ such that for any $\delta \in (0,0.1)$, we have 
\[ \Big| \|\hat{Q}(\theta^*) - \theta^*\|_2^2 - \EE \big[ \|\hat{Q}(\theta^*) - \theta^*\|_2^2 \big] \Big| \le  C \sigma^2 \bigg( \frac{d^{3/2}}{N}\sqrt{\log\frac{1}{\delta}}+ \frac{d}{N} \log \frac{1}{\delta} \bigg) \]
with (conditional) probability at least $1 - \delta$.
\end{lemma}

\begin{proof}
Let $\eps' := \eps/\sigma \sim \mathcal{N}(0, I_N)$, and then $y_r = x_r^\top\theta^* + \sigma\eps'_r$.
Consider the random vector
\[f(\eps')  :=  \hat{Q}(\theta^*) - \theta^* = \hat\Sigma^\dag\left(\frac{d-1}{2N}\sum\limits_{r = 1}^N\tanh\left(\frac{(x_r^\top\theta^* + \sigma \eps'_r) x_r^\top\theta^*}{\sigma^2}\right) (x_r^\top\theta^* + \sigma \eps'_r) x_r\right) - \theta^*.\]
We claim that the function $f(\cdot)$ is $15 \sigma \sqrt{d/N}$-Lipschitz in the Euclidean distance. 

Assuming the claim for a moment, we now show that the random vector $f(\eps')$ satisfies the convex concentration property in Definition \ref{def:conv_conc_prop}. To this end, let $\phi$ be an arbitrary $1$-Lipschitz function. (Note that we do not require $\phi$ to be convex and hence prove something stronger.)
Then $\phi \circ f$ is $15 \sigma \sqrt{d/N}$-Lipschitz, so by Lemma~\ref{lem:ver-5.2.2},
\[
\p\left\{ \big| \phi(f(\eps')) - \EE[\phi(f(\eps'))] \big| \ge t \right\} \leq 2 \exp\left( \frac{-t^2}{C_1^2 \sigma^2 d/N} \right) \]
for an absolute constant $C_1 > 0$.  Therefore, $f(\eps')$ satisfies the convex concentration property for $K = C_1\,\sigma\,\sqrt{d/N}$. 
We have seen in \eqref{eq:em-expect} that $\EE[f(\eps')] = 0$. 
Then Lemma~\ref{lem:conv_conc} with $A = I_d$ yields that for an absolute constant $C_2 > 0$,
$$
\Pr\left\{ \big| \|\hat{Q}(\theta^*) - \theta^*\|_2^2 - \EE[\|\hat{Q}(\theta^*) - \theta^*\|_2^2] \big| \geq t \right\} \\
\leq 2\exp\left(-\frac{1}{C_2} \min\left(\frac{N^2t^2}{\sigma^4d^3}, \frac{Nt}{\sigma^2d}\right)\right).
$$
From here, the lemma follows.

It remains to  prove the claim that the function $f(\cdot)$ is $15 \sigma \sqrt{d/N}$-Lipschitz. 
Let $\eps', \eps'' \in \R^N$.
By Taylor's theorem, we have for some $t_r \in (0,1)$, 
\begin{align*}
&\tanh\left(\frac{(x_r^\top\theta^*+ \sigma\eps''_r)x_r^\top\theta^*}{\sigma^2}\right)(x_r^\top\theta^*+ \sigma\eps''_r) 
- \tanh\left(\frac{(x_r^\top\theta^*+ \sigma\eps'_r)x_r^\top\theta^*}{\sigma^2}\right)(x_r^\top\theta^*+ \sigma\eps'_r) \\
&= (\eps''_r - \eps'_r)\frac{d}{d\tilde\eps}\left(\tanh\left(\frac{(x_r^\top\theta^*+ \sigma\tilde\eps)x_r^\top\theta^*}{\sigma^2}\right)(x_r^\top\theta^*+ \sigma\tilde\eps)\right)\bigg|_{\tilde\eps = \eps'_r+t_r(\eps''_r - \eps'_r)} \\
&= (\eps''_r - \eps'_r) \, \sigma \, h\left(\frac{(x_r^\top\theta^*+ \sigma\tilde\eps_r)x_r^\top\theta^*}{\sigma^2} \right) ,
\end{align*}
where $h(t) := \tanh(t) + t \tanh'(t)$ for $t \in \R$,
and $\tilde\eps_r := \eps'_r+t_r(\eps''_r - \eps'_r)$.  
Then we have 
$$
f(\eps'') - f(\eps')
= \hat\Sigma^\dag \left( \frac{d-1}{2N}\sum_{r = 1}^N (\eps''_r - \eps'_r) \, \sigma \, h\left(\frac{(x_r^\top\theta^*+ \sigma\tilde\eps_r) x_r^\top\theta^*}{\sigma^2} \right) x_r\right)  
= \sigma \hat\Sigma^\dag \tilde X  D (\eps'' - \eps') ,
$$
where $\tilde X \in \R^{d\times N}$ is the matrix whose $r$th column is $\frac{d-1}{2N}\,x_r$, and $D \in \R^{N\times N}$ is the diagonal matrix defined by 
\[D_{rr} := h\left(\frac{(y_r+t_r(y_r - y_r))x_r^\top\theta^*}{\sigma^2}\right).\]
As a result,
$$
\|f(\eps'') - f(\eps')\|_2 
\leq \sigma\|\hat\Sigma^\dag\|_{op}\|\tilde X\|_{op}\|D\|_{op}\|\eps'-\eps''\|_2 .
$$
Note that $\tilde X \tilde X^\top = \frac{d-1}{2N}\hat\Sigma$ and so we have
\[\|\tilde X\|_{op} 
= \sqrt{\frac{d-1}{2N} \|\hat \Sigma\|_{op}}
\le \sqrt{\frac{3d}{2N}} \]
by the conditioning. 
Moreover, we have $\|\hat\Sigma^\dag\|_{op} \le 5$ by the conditioning, and $\|D\|_{op} \leq 2$ by the fact that $|h(t)| = |\tanh(t) + t \tanh'(t)| \le 2$ for all $t \in \R$. Combining all the bounds gives
\[\|f(\eps'') - f(\eps')\|_2 \leq 15 \,\sigma\,\sqrt{\frac{d}{N}} \, \|\eps'' -  \eps'\|_2 ,\]
completing the proof.
\end{proof}

\subsection{Convergence analysis}
\label{sec:convergence}

Assuming the bounds in Lemma~\ref{lem:high-prob-bds} and the control on $\|\hat \Sigma^\dag\|_\infty$ in Proposition~\ref{prop:sigma-dag-bound}, the following result shows that the EM operator is contractive locally around $\theta^*$, deterministically.

\begin{lemma}
\label{lem:easy-contraction}
There exist absolute constants $C, C' > 0$ such that the following holds.
Suppose that all the bounds in Lemma~\ref{lem:high-prob-bds} hold with $C>0$ and $\delta \in (0,0.1)$.
Consider $\theta, \theta' \in \cH$, $\lambda \in (0,1)$, and $\Delta := \max\{\frac{1}{\lambda} \|\theta - \theta'\|_\infty, \frac{3 \sigma}{\sqrt{\lambda}}\}$, such that
$$
\max\{ \|\theta - \theta^*\|_\infty, \, \|\theta' - \theta^*\|_\infty \} \le \frac{\Delta}{6} , \qquad
N \ge C' \frac{d}{\lambda^2} \log \frac d \delta , \qquad
\max_{i \in [d]} |S_i(\Delta)| \le \frac{\lambda^2}{C'} d ,
$$
where $S_i(\cdot)$ is defined as in \eqref{def:S_i}.
Then we have
$$
\|\bar Q(\theta) - \bar Q(\theta')\|_\infty
\le \lambda \|\theta - \theta'\|_\infty .
$$
If, in addition, $\|\hat \Sigma^\dag\|_\infty \le \frac{1}{2\lambda}$, then
$$
\|\hat Q(\theta) - \hat Q(\theta')\|_\infty
\le \frac 12 \|\theta - \theta'\|_\infty .
$$
\end{lemma}


\begin{proof}
For notational simplicity, let us consider iteration of the first entry of each vector and bound $|\bar Q(\theta)_1 - \bar Q(\theta')_1|$. The same analysis applies to any other entry with straightforward modification.
We can write the iteration of the first entry as
\[\bar Q(\theta)_1 = \frac{d-1}{2N} \sum_{r \in \cR_1} \tanh \Big( \frac{ y_r \, x_r^\top \theta }{\sigma^2} \Big) y_r, \]
where $\cR_1$ is defined as in \eqref{eq:def-ri}.
It follows that
\begin{align*}
|\bar Q(\theta)_1 - \bar Q(\theta')_1|
\leq \frac{d-1}{2N} \sum_{r \in \cR_1} \left|\tanh \Big( \frac{ y_r \, x_r^\top \theta }{\sigma^2} \Big) - \tanh \Big( \frac{ y_r \, x_r^\top \theta' }{\sigma^2} \Big)\right| |y_r|.
\end{align*}
Define $\theta(z) := \theta' + z(\theta - \theta')$ for $z \in [0,1]$. By the fundamental theorem of calculus, we have
\begin{align*}
\left| \tanh \bigg( \frac{ y_r \, x_r^\top \theta }{\sigma^2} \bigg) - \tanh \bigg( \frac{ y_r \, x_r^\top \theta' }{\sigma^2} \bigg) \right|
&= \left| \int_{0}^1\tanh'\left( \frac{ y_r \, x_r^\top \theta(z) }{\sigma^2} \right)\,\left( \frac{ y_r \, x_r^\top (\theta - \theta') }{\sigma^2} \right)\,dz \right| \\
&\leq \frac{ | y_r \, x_r^\top (\theta - \theta') |}{\sigma^2}\int_0^1 4\exp {\left( -2\frac{ |y_r \, x_r^\top \theta(z)| }{\sigma^2} \right)}\,dz \\
&\leq 4\frac{ |y_r \, x_r^\top (\theta - \theta')| }{\sigma^2}\exp {\left( -2\min_{0 \leq z \leq 1}\frac{ |y_r \, x_r^\top \theta(z)| }{\sigma^2} \right)} ,
\end{align*}
where we used the fact $0 \le \tanh'(t) = (\cosh(t))^{-2} \le 4 \exp(-2|t|)$ for all $t \in \R$.
Combining the above bounds with the trivial bound $|\tanh(t)| \le 1$ for $t \in \R$, we obtain
\begin{equation}
|\bar Q(\theta)_1 - \bar Q(\theta')_1| \leq \frac{d}{N} \sum_{r \in \cR_1} \min\left\{ |y_r| , \, 2 \frac{y_r^2 \, |x_r^\top(\theta - \theta')|}{\sigma^2}\,\exp {\left( -2\min_{0 \leq z \leq 1}\frac{ |y_r \, x_r^\top \theta(z)| }{\sigma^2} \right)} \right\} .
\label{eq:int-bd}
\end{equation}

For brevity, we let $\eta := \|\theta - \theta'\|_\infty$ in the sequel.
Then $\Delta = \max\{\frac \eta \lambda, \frac{3\sigma}{\sqrt{\lambda}}\}$ by definition.
Recall that the sets $S_1(\Delta)$ and $\cR_1(\Delta)$ are defined in \eqref{def:S_i} and \eqref{def:R_i} respectively.
We now analyze the summands in \eqref{eq:int-bd} for $r \in \cR_1(\Delta)$ and $r \in \cR_1 \setminus \cR_1(\Delta)$ separately.

First, for $r \in \cR_1 \setminus \cR_1(\Delta)$, we have $x_r = e_1 - e_j$ such that $|x_r^\top \theta^*| = |\theta_1^* - \theta_j^*| > \Delta$.
Moreover, $\|\theta - \theta^*\|_\infty \le \Delta/6$ and $\|\theta' - \theta^*\|_\infty \le \Delta/6$ by assumption.
It follows that
\begin{align*}
|x_r^\top\theta(z)| = |\theta(z)_1 - \theta(z)_j|
&\ge |\theta^*_1 - \theta^*_j| - |\theta(z)_1 - \theta^*_1| - |\theta(z)_j - \theta^*_j| \\
&\ge |\theta^*_1 - \theta^*_j| - z |\theta_1 - \theta^*_1| - (1-z) |\theta'_1 - \theta^*_1| - z |\theta_j - \theta^*_j| - (1-z) |\theta'_j - \theta^*_j| \\
&\ge |x_r^\top \theta^*| - \frac{\Delta}{3}
> \frac 23 |x_r^\top\theta^*|
\end{align*}
for any $z \in [0,1]$.
Moreover, it holds that $|x_r^\top(\theta - \theta')| \le 2 \|\theta - \theta'\|_\infty = 2 \eta$.
Therefore,
$$
2 \frac{y_r^2 \, |x_r^\top(\theta - \theta')|}{\sigma^2}\,\exp {\left( -2\min_{0 \leq z \leq 1}\frac{ |y_r \, x_r^\top \theta(z)| }{\sigma^2} \right)}
\leq 4 \eta \frac{y_r^2}{\sigma^2}\,\exp {\left( \frac{ -4 \, |y_r| \, |x_r^\top \theta^*| }{3 \sigma^2} \right)} .
$$
It is not hard to see that $\max_{a \ge 0} a^2 \exp(-ab) = 4/(eb)^2$ for $b>0$.
Taking $a = |y_r|$ and $b = \frac{ 4 \, |x_r^\top \theta^*| }{3 \sigma^2}$ in the above bound then yields
\begin{align*}
2 \frac{y_r^2 \, |x_r^\top(\theta - \theta')|}{\sigma^2}\,\exp {\left( -2\min_{0 \leq z \leq 1}\frac{ |y_r \, x_r^\top \theta(z)| }{\sigma^2} \right)}
\le \frac{9 \sigma^2}{e^2 \, |x_r^\top \theta^*|^2} \eta
\le \Big( \frac{3 \sigma}{e \Delta} \Big)^2 \eta .
\end{align*}
As a result,
\begin{equation}
\frac{d}{N} \sum_{r \in \cR_1 \setminus \cR_1(\Delta)} 2 \frac{y_r^2 \, |x_r^\top(\theta - \theta')|}{\sigma^2}\,\exp {\left( -2\min_{0 \leq z \leq 1}\frac{ |y_r \, x_r^\top \theta(z)| }{\sigma^2} \right)}
\le \frac dN |\cR_1| \Big( \frac{3 \sigma}{e \Delta} \Big)^2 \eta
\le \frac{\lambda}{2} \eta
\label{eq:good-part}
\end{equation}
by the bound $|\cR_1| \le 3N/d$ in Lemma~\ref{lem:high-prob-bds} and the condition $\Delta \ge \frac{3 \sigma}{\sqrt{\lambda}}$.

Second, for $r \in \cR_1(\Delta)$, since $\exp(-t) \le 1$ for $t \ge 0$ and $|x_r^\top(\theta - \theta')| \le 2 \eta$, we have
\begin{align*}
&\frac{d}{N} \sum_{r \in \cR_1(\Delta)} \min\left\{ |y_r| , \, 2 \frac{y_r^2 \, |x_r^\top(\theta - \theta')|}{\sigma^2}\,\exp {\left( -2\min_{0 \leq z \leq 1}\frac{ |y_r \, x_r^\top \theta(z)| }{\sigma^2} \right)} \right\} \\
&\le \frac{d}{N} \sum_{r \in \cR_1(\Delta)} \min\left\{ |y_r| , \, 4 \eta \frac{y_r^2}{\sigma^2}  \right\}
\le \frac{d}{N}  \min\bigg\{ \sum_{r \in \cR_1(\Delta)} |y_r| , \, \frac{4 \eta}{\sigma^2} \sum_{r \in \cR_1(\Delta)} y_r^2 \bigg\} .
\end{align*}
Plugging the bounds from Lemma~\ref{lem:high-prob-bds} into the above then yields
\begin{align}
&\frac{d}{N} \sum_{r \in \cR_1(\Delta)} \min\left\{ |y_r| , \, 2 \frac{y_r^2 \, |x_r^\top(\theta - \theta^*)|}{\sigma^2}\,\exp {\left( -2\min_{0 \leq z \leq 1}\frac{ |y_r \, x_r^\top \theta(z)| }{\sigma^2} \right)} \right\} \notag \\
&\le C \frac{d}{N}  \bigg( |S_1(\Delta)| \frac{N}{d^2} + \log\frac d\delta \bigg) \min\left\{ \Delta , \, \frac{4 \eta}{\sigma^2} \Delta^2  \right\} .
\label{eq:bad-part}
\end{align}

Lastly, combining \eqref{eq:good-part} and \eqref{eq:bad-part} with \eqref{eq:int-bd} gives
$$
|\bar Q(\theta)_1 - \bar Q(\theta')_1|
\leq \frac{\lambda}{2} \eta + C \bigg( \frac{|S_1(\Delta)|}{d} + \frac dN \log\frac d\delta \bigg) \min\bigg\{ \Delta , \, \frac{4 \eta}{\sigma^2} \Delta^2  \bigg\} .
$$
Since $\Delta = \max\{ \frac \eta \lambda , \frac{3\sigma}{\sqrt{\lambda}} \}$, we have
$$
\min\bigg\{ \Delta , \, \frac{4 \eta}{\sigma^2} \Delta^2  \bigg\}
\le \max\bigg\{ \frac \eta \lambda , \, \frac{4 \eta}{\sigma^2} \Big( \frac{3\sigma}{\sqrt{\lambda}} \Big)^2  \bigg\}
\le 36 \frac \eta \lambda.
$$
The above two bounds together with the assumptions $N \ge C' \frac{d}{\lambda^2} \log \frac d \delta$ and $|S_1(\Delta)| \le \frac{\lambda^2}{C'} d$ yield
$$
|\bar Q(\theta)_1 - \bar Q(\theta')_1|
\leq \frac{\lambda}{2} \eta + C \bigg( \frac{\lambda^2}{C'} + \frac{\lambda^2}{C'} \bigg) \cdot 64 \frac \eta \lambda \le \lambda \eta ,
$$
once $C'$ is chosen to be sufficiently large.

For the final statement of the lemma, note that $\hat Q(\theta) = \hat \Sigma^\dag \bar Q(\theta)$, so we have
\[\|\hat Q(\theta) - \hat Q(\theta')\|_\infty \leq \|\hat \Sigma^\dag\|_\infty\|\bar Q(\theta) - \bar Q(\theta')\|_\infty,\]
from which the conclusion follows.
\end{proof}

We summarize the convergence results for the EM sequence in the following proposition.

\begin{proposition}
\label{prop:em-contraction}
There exist absolute constants $C, C' > 0$ such that the following holds for any $\delta \in (0,0.1)$ and any integer $L \ge 2$.
Fix $\theta^{(0)} \in \cH$ and let $\Delta^{(0)} := \max\{ 4L \|\theta^{(0)} - \theta^*\|_\infty, 6 \sqrt{L} \, \sigma \}$. 
Suppose that
\begin{subequations}
\label{eq:cond-samp-init}
\begin{align}
&N \ge C \max \Big\{ L^2 d^{\frac{L}{L-1}} \log \frac{dL}{\delta} , L^3 d \log \frac{L}{\delta} \Big\} , \label{eq:cond-1-sample-size} \\
&\max_{i \in [d]} |S_i(\Delta^{(0)})| \le \frac{d}{C L^2} , \label{eq:cond-2-s-i-delta}
\end{align}
\end{subequations}
where $S_i(\cdot)$ is defined in \eqref{def:S_i}.
Let $\{\theta^{(t)}\}_{t \ge 0}$ be the EM iterates defined in \eqref{eq:iteration-em}. 
Moreover, let 
$$
\tau := C' \sigma \sqrt{\frac{L\,d}{N}\log\frac d\delta} , \qquad
T := \max\bigg\{ 0, \bigg\lceil \log_{4/3} \bigg( \frac{\|\theta^{(0)} - \theta^*\|_\infty}{4\tau} \bigg) \bigg\rceil \bigg\} . 
$$
Then it holds with probability at least $1-\delta$ that 
\begin{itemize}
\item
there exists $\hat \theta \in \cH$ such that $\hat Q(\hat \theta) = \hat \theta$ and $\|\hat \theta - \theta^*\|_\infty \le 4 \tau$;

\item
the sequence $\{\theta^{(t)}\}_{t \ge 0}$ converges to $\hat \theta$;

\item
$\| \theta^{(T+t)} - \hat \theta \, \|_\infty \le 8 \tau / 2^t$ for all $t \ge 0$.
\end{itemize}
\end{proposition}

\begin{proof}
By Proposition~\ref{prop:sigma-dag-bound} and the assumption on $N$, it holds with probability at least $1-\delta/3$ that
$$
\| \hat \Sigma^\dag \|_\infty \le 2 L .
$$
By Lemma~\ref{lem:theta_infty_bd}, it holds with probability at least $1-\delta/3$ that 
\begin{equation}
\|\hat{Q}(\theta^*) - \theta^*\|_\infty \leq C' \sigma \sqrt{\frac{L\,d}{N}\log\frac d\delta} = \tau
\label{eq:bd-tau}
\end{equation}
for a constant $C' > 0$. 
Applying Lemma~\ref{lem:high-prob-bds} and Lemma~\ref{lem:easy-contraction} with $\lambda = \frac{1}{4L}$, we see that with probability at least $1 - \delta/3$, the following holds:
For all $\theta, \theta' \in \cH$ and $\Delta := \max\{ 4L \|\theta - \theta'\|_\infty, 6 \sqrt{L} \, \sigma \}$ such that
\begin{equation}
\max\{ \|\theta - \theta^*\|_\infty, \, \|\theta' - \theta^*\|_\infty \} \le \frac{\Delta}{6} ,
\qquad
\max_{i \in [d]} |S_i(\Delta)| \le \frac{d}{C L^2} ,
\label{eq:int-assum}
\end{equation}
we have
\begin{equation}
\|\hat Q(\theta) - \hat Q(\theta')\|_\infty
\le \frac 12 \|\theta - \theta'\|_\infty .
\label{eq:int-concl}
\end{equation}
In the sequel, we condition on the event of probability at least $1-\delta$ that all the above bounds hold. 
We split the rest of the proof into two parts as we will use the above conclusion twice.

\paragraph{Part I:} Convergence to a neighborhood of the ground truth. 
Let us take $\theta = \theta^{(t)}$ for $t \ge 0$ and $\theta' = \theta^*$ in \eqref{eq:int-assum} and \eqref{eq:int-concl}.
Let $\Delta^{(t)} := \max\{ 4L \|\theta^{(t)} - \theta^*\|_\infty, 6 \sqrt{L} \, \sigma \}$. 
Note that we have $\|\theta^{(t)} - \theta^*\|_\infty \le \frac{4L}{6} \|\theta^{(t)} - \theta^*\|_\infty \le \frac{\Delta^{(t)}}{6}$, so the first condition in \eqref{eq:int-assum} holds. 
We now show inductively that, for $t \ge 0$, 
\begin{equation}
\Delta^{(t)} \le \Delta^{(0)} ,
\label{eq:ind-cond-1}
\end{equation}
and furthermore,
\begin{equation}
\| \theta^{(t+1)} - \theta^* \|_\infty \le \frac 12 \|\theta^{(t)} - \theta^*\|_\infty + \tau .
\label{eq:ind-cond-2}
\end{equation}

First, \eqref{eq:ind-cond-1} is trivial for $t = 0$. 
Suppose that \eqref{eq:ind-cond-1} holds. 
Then, by $\max_{i \in [d]} |S_i(\Delta^{(0)})| \le \frac{d}{C L^2}$ and that $S_i(\cdot)$ is a nondecreasing function defined in \eqref{def:S_i}, we see that $\max_{i \in [d]} |S_i(\Delta^{(t)})| \le \frac{d}{C L^2}$. 
Hence \eqref{eq:int-assum} is satisfied, and then \eqref{eq:int-concl} implies
$$
\| \theta^{(t+1)} - \hat Q(\theta^*)\|_\infty \le \frac 12 \|\theta^{(t)} - \theta^*\|_\infty .
$$
Combining this bound with \eqref{eq:bd-tau} gives \eqref{eq:ind-cond-2}. 

Next, suppose that \eqref{eq:ind-cond-2} holds. 
If $\|\theta^{(t)} - \theta^*\|_\infty \ge 2 \tau$, then \eqref{eq:ind-cond-2} implies that $\| \theta^{(t+1)} - \theta^* \|_\infty \le \| \theta^{(t)} - \theta^* \|_\infty$. 
As a result, $\Delta^{(t+1)} \le \Delta^{(t)}$, so \eqref{eq:ind-cond-1} holds for $t+1$ in place of $t$. 
On the other hand, if $\|\theta^{(t)} - \theta^*\|_\infty \le 2 \tau$, then \eqref{eq:ind-cond-2} implies that $\| \theta^{(t+1)} - \theta^* \|_\infty \le 2 \tau$. 
By the definition of $\tau$ and our assumption on $N$, it is clear that $4L \| \theta^{(t+1)} - \theta^* \|_\infty \le 8 L \tau \le 6 \sqrt{L} \, \sigma$. 
As a result, $\Delta^{(t+1)} \le 6 \sqrt{L} \, \sigma \le \Delta^{(0)}$, so again \eqref{eq:ind-cond-1} holds for $t+1$ in place of $t$. 
This completes the induction.

To conclude this part, note that by \eqref{eq:ind-cond-2}, the EM iterates $\{\theta^{(t)}\}_{t \ge 0}$ satisfy
\begin{equation}
\| \theta^{(t+1)} - \theta^* \|_\infty \le \frac 34 \|\theta^{(t)} - \theta^*\|_\infty 
\label{eq:part-1-concl}
\end{equation}
if $\|\theta^{(t)} - \theta^*\|_\infty > 4 \tau$.
Moreover, define 
$$
\cB := \{ \theta \in \cH : \|\theta - \theta^*\|_\infty \le 4 \tau \} , \qquad
T_1 := \min \big\{ t \ge 0 : \theta^{(t)} \in \cB \big\} .
$$
By \eqref{eq:part-1-concl}, we have
$$
T_1 \le \max\bigg\{ 0, \bigg\lceil \log_{4/3} \bigg( \frac{\|\theta^{(0)} - \theta^*\|_\infty}{4\tau} \bigg) \bigg\rceil \bigg\} = T .
$$
Finally, $\theta^{(t)} \in \cB$ for all $t \ge T_1$ by \eqref{eq:ind-cond-2}.

\paragraph{Part II:} Convergence to a fixed point. 
We now focus on the set $\cB$ defined above. 
Fix any $\theta, \theta' \in \cB$. 
We have $\|\theta - \theta'\|_\infty \le 8 \tau$. 
By the definition of $\tau$ and our assumption on $N$, it holds that $4L \| \theta - \theta' \|_\infty \le 16 L \tau \le 6 \sqrt{L} \, \sigma$. 
Hence, $\Delta = \max\{ 4L \|\theta - \theta'\|_\infty, 6 \sqrt{L} \, \sigma \} = 6 \sqrt{L} \, \sigma \le \Delta^{(0)}$. 
By our assumption on $\Delta^{(0)}$ and that $S_i(\cdot)$ is a nondecreasing function, we obtain $\max_{i \in [d]} |S_i(\Delta)| \le \frac{d}{C L^2}$. 
In addition, $\max\{ \|\theta - \theta^*\|_\infty, \, \|\theta' - \theta^*\|_\infty \} \le 4 \tau \le \sqrt{L} \, \sigma = \frac{\Delta}{6}$.
Therefore, \eqref{eq:int-assum} is satisfied, and so \eqref{eq:int-concl} holds.

Similar to Part~I, taking $\theta' = \theta^*$ in \eqref{eq:int-concl}, we obtain
$$
\|\hat Q(\theta) - \theta^*\|_\infty 
\le \|\hat Q(\theta) - \hat Q(\theta^*)\|_\infty + \|\hat Q(\theta^*) - \theta^*\|_\infty
\le \frac 12 \|\theta - \theta^*\|_\infty + \tau 
\le 4 \tau ,
$$
so the EM operator $\hat Q$ can be viewed as a map on $\cB$. 
Furthermore, \eqref{eq:int-concl} says that $\hat Q$ is a contraction on $\cB$. 
By the Banach fixed-point theorem, $\hat Q$ has a unique fixed point $\hat \theta$ in $\cB$, and the EM sequence $\{\theta^{(t)}\}_{t \ge 0}$ converges to $\hat \theta$ with
$$
\|\theta^{(t+1)} - \hat \theta \, \|_\infty
\le \frac 12 \, \|\theta^{(t)} - \hat \theta \, \|_\infty 
$$
for all $t \ge T_1$. 
Since $\|\theta^{(T_1)} - \hat \theta\|_\infty \le 8 \tau$, the conclusion follows immediately.
\end{proof}

\subsection{Sharp results in the low-noise regime}
\label{sec:low-noise}

Before proving the sharp bound on $\|\hat \theta - \theta^*\|_2$, let us start with two lemmas that control the lower order terms.

\begin{lemma}
\label{lem:A-infty-norm-bd}
There exist absolute constants $C, C' > 0$ such as the following holds for any $\delta \in (0,0.1)$. 
Suppose that all the bounds in Lemma~\ref{lem:high-prob-bds} hold with the constant $C$, and that $|y_r - x_r^\top \theta^*| \le C \sigma \sqrt{\log \frac{N}{\delta}}$ for all $r \in [N]$. 
Define
\begin{equation}
A := \frac{d-1}{2N}\sum_{r = 1}^N\frac{y_r^2}{\sigma^2}\tanh' \bigg(\frac{y_r x_r^\top \theta^*}{\sigma^2} \bigg) x_r x_r^\top .
\label{eq:def-A}
\end{equation}
Moreover, let $\Delta := 2 C \sigma \sqrt{\log \frac{N}{\delta}}$, and define $S_i(\Delta)$ as in \eqref{def:S_i}. 
Then we have
$$
\|A\|_\infty \le C' \bigg( \max_{i \in [d]} \frac{|S_i(\Delta)|}{d} + \frac{d}{N}  \log\frac d\delta \bigg) \log \frac{N}{\delta} .
$$
\end{lemma}

\begin{proof}
We have $\|A\|_\infty = \max_{i \in [d]} \sum_{j=1}^d |A_{ij}|$. 
If $x_r = e_i - e_j$, then $x_r x_r^\top$ is the matrix with entries $(x_r x_r^\top)_{ii} = (x_r x_r^\top)_{jj} = 1$, $(x_r x_r^\top)_{ij} = (x_r x_r^\top)_{ji} = -1$, and $0$ elsewhere. 
Also, note that $\frac{y_r^2}{\sigma^2}\tanh' \big(\frac{y_r x_r^\top \theta^*}{\sigma^2} \big) \ge 0$.
Then, by the definition of $A$, it is not hard to see that 
\begin{equation}
\sum_{j=1}^d |A_{ij}|
= \frac{d-1}{N}\sum_{r \in \cR_i} \frac{y_r^2}{\sigma^2}\tanh' \bigg(\frac{y_r x_r^\top \theta^*}{\sigma^2} \bigg) 
\le \frac{4d}{N}\sum_{r \in \cR_i} \frac{y_r^2}{\sigma^2} \exp \bigg(-2 \, \bigg| \frac{y_r x_r^\top \theta^*}{\sigma^2} \bigg| \bigg) 
\label{eq:A-row-sum-bd}
\end{equation}
where $\cR_i$ is defined in \eqref{eq:def-ri}, and the inequality follows from the fact $0 \le \tanh'(t) = (\cosh(t))^{-2} \le 4 \exp(-2|t|)$ for all $t \in \R$.

For $\Delta = 2 C \sigma \sqrt{\log \frac{N}{\delta}}$, let $\cR_i(\Delta)$ be defined in \eqref{def:R_i}. 
Similar to the proof of Lemma~\ref{lem:easy-contraction}, we split the analysis of the sum in \eqref{eq:A-row-sum-bd} into two parts. 
First, let us consider $r \in \cR_i \setminus \cR_i(\Delta)$. 
Then $|x_r^\top \theta^*| > \Delta$, and by the assumption $|y_r - x_r^\top \theta^*| \le C \sigma \sqrt{\log \frac{N}{\delta}} = \Delta/2$, we have
$$
\bigg| \frac{y_r x_r^\top \theta^*}{\sigma^2} \bigg| 
\ge \frac{(x_r^\top \theta^*)^2}{2 \sigma^2} .
$$
Consequently,
\begin{align*}
\frac{4d}{N} \sum_{r \in \cR_i \setminus \cR_i(\Delta)} \frac{y_r^2}{\sigma^2} \exp \bigg(-2 \, \bigg| \frac{y_r x_r^\top \theta^*}{\sigma^2} \bigg| \bigg) 
&\le \frac{4d}{N} \, N \, \frac{4 (x_r^\top \theta^*)^2}{\sigma^2} \exp \bigg( - \frac{(x_r^\top \theta^*)^2}{\sigma^2} \bigg) \\
&\stackrel{(a)}{\le} 16 \, d \, \frac{\Delta^2}{\sigma^2} \exp \bigg( - \frac{\Delta^2}{\sigma^2} \bigg)
\stackrel{(b)}{\le} \frac{1}{N^{10}} ,
\end{align*}
where $(a)$ holds as the function $t \exp(-t)$ is decreasing for $t \ge 1$, and $(b)$ holds by the definition of $\Delta$ provided that the constant $C$ is sufficiently large. 

Next, consider $r \in \cR_i(\Delta)$. We have 
\begin{align*}
\frac{4d}{N}\sum_{r \in \cR_i(\Delta)} \frac{y_r^2}{\sigma^2} \exp \bigg(-2 \, \bigg| \frac{y_r x_r^\top \theta^*}{\sigma^2} \bigg| \bigg)
&\le \frac{4d}{\sigma^2 N} \sum_{r \in \cR_i(\Delta)} y_r^2 \\
&\le \frac{4d}{\sigma^2 N} C \Delta^2 \left( |S_i(\Delta)| \frac{N}{d^2} + \log\frac d\delta \right) \\
&= 16 C^3 \bigg( \frac{|S_i(\Delta)|}{d} + \frac{d}{N}  \log\frac d\delta \bigg) \log \frac{N}{\delta} ,
\end{align*}
where the second inequality follows from Lemma~\ref{lem:high-prob-bds}. 

Finally, combining the above two bounds with \eqref{eq:A-row-sum-bd} finishes the proof.
\end{proof}

\begin{lemma}\label{lem:taylor_error}
There exist absolute constants $C, C' > 0$ such as the following holds for any $\delta \in (0,0.1)$. 
Suppose that all the bounds in Lemma~\ref{lem:high-prob-bds} hold with the constant $C$, and that $|y_r - x_r^\top \theta^*| \le C \sigma \sqrt{\log \frac{N}{\delta}}$ for all $r \in [N]$. 
For $\theta \in \cH$, $t \in [0,1]^d$, and $r \in [N]$, define
\begin{subequations}
\begin{align}
\omega(\theta,r,t_r) &:= \frac{y_r^2}{2 \sigma^4}\,\tanh''\left(\frac{y_r x_r^\top (\theta^* + t_r (\theta - \theta^*))}{\sigma^2}\right) (x_r^\top(\theta - \theta^*))^2 , 
\label{eq:err-term} \\
\omega(\theta,t) &:= \frac{d-1}{2N}\sum_{r = 1}^N\omega(\theta,r,t_r)\,y_r\,x_r .
\label{eq:err-term-sum}
\end{align}
\label{eq:err-term-all}
\end{subequations}
Suppose further that $\|\theta - \theta^*\|_\infty \le \frac{C}{4} \sigma \sqrt{\log \frac{N}{\delta}}$. 
Then we have
\[\|\omega(\theta,t)\|_\infty \le \frac{C'}{\sigma} \|\theta - \theta^*\|_\infty^2 \Big(\log \frac{N}{\delta}\Big)^{3/2} .\]
\end{lemma}

\begin{proof}
Fix $i \in [d]$. 
By \eqref{eq:err-term-all} and the definition of $\cR_i$ in \eqref{def:R_i}, we have 
$$
|\omega(\theta,t)_i| 
\le \frac{d-1}{2N} \sum_{r \in \cR_i} |\omega(\theta, r, t_r) \, y_r| 
\le \frac{d}{4N} \sum_{r \in \cR_i} \frac{|y_r^3|}{\sigma^4} \left| \tanh''\left(\frac{y_r x_r^\top (\theta^* + t_r (\theta - \theta^*))}{\sigma^2}\right) \right| (x_r^\top(\theta - \theta^*))^2 .
$$
Using that $\left| \tanh''\left(t\right) \right| \le 8 \exp(-2|t|)$ for all $t \in \R$ and $(x_r^\top(\theta - \theta^*))^2 \le 4 \|\theta - \theta^*\|_\infty^2$, we obtain 
\begin{equation}
|\omega(\theta,t)_i| 
\le \frac{8 d}{\sigma N} \|\theta - \theta^*\|_\infty^2 \sum_{r \in \cR_i} \frac{|y_r^3|}{\sigma^3} \exp \left( - 2 \, \left| \frac{y_r x_r^\top (\theta^* + t_r (\theta - \theta^*))}{\sigma^2} \right| \right) .
\label{eq:omega-entry-bd}
\end{equation}

Similar to the proof of the previous lemma, we again split the analysis into two parts. 
Define $\Delta := 2 C \sigma \sqrt{\log \frac{N}{\delta}}$. 
First, consider $r \in \cR_i \setminus \cR_i(\Delta)$. 
Then $|x_r^\top \theta^*| > \Delta$, and by the assumptions $|y_r - x_r^\top \theta^*| \le \Delta/2$ and $\|\theta - \theta^*\|_\infty \le \Delta/4$, we have $t_r |x_r^\top (\theta - \theta^*)| \le 2 \|\theta - \theta^*\|_\infty \le \Delta/2$ and 
$$
\bigg| \frac{y_r x_r^\top (\theta^* + t_r (\theta - \theta^*))}{\sigma^2} \bigg| 
\ge \frac{(x_r^\top \theta^*)^2}{4 \sigma^2} .
$$
It follows that 
$$
\frac{|y_r^3|}{\sigma^3} \exp \left( - 2 \, \left| \frac{y_r x_r^\top (\theta^* + t_r (\theta - \theta^*))}{\sigma^2} \right| \right) 
\le 8 \frac{|x_r^\top \theta^*|^3}{\sigma^3} \exp \left( - \frac{(x_r^\top \theta^*)^2}{2 \sigma^2} \right) 
\le 16 
$$
since the function $t^3 \exp(-t^2/2) \le 2$ for $t \in \R$. 
Next, consider $r \in \cR_i(\Delta)$. 
Then $|x_r^\top\theta^*| \le \Delta$ and so $|y_r| \le 2 \Delta$. 
Hence we have 
\[\frac{|y_r^3|}{\sigma^3} \exp \left( - 2 \, \left| \frac{y_r x_r^\top (\theta^* + t_r (\theta - \theta^*))}{\sigma^2} \right| \right) 
\le 8 \frac{\Delta^3}{\sigma^3} 
= 64 C^3 \Big(\log \frac{N}{\delta}\Big)^{3/2} .\]

Putting the two cases together with \eqref{eq:omega-entry-bd}, we obtain
$$
|\omega(\theta,t)_i| 
\le \frac{8 d}{\sigma N} \|\theta - \theta^*\|_\infty^2 \, |\cR_i| \cdot 64 C^3 \Big(\log \frac{N}{\delta}\Big)^{3/2} 
\le \frac{C'}{\sigma} \|\theta - \theta^*\|_\infty^2 \Big(\log \frac{N}{\delta}\Big)^{3/2} ,
$$
where the second inequality follows from Lemma~\ref{lem:high-prob-bds}. 
This holds for any $i \in [d]$, so the proof is complete.
\end{proof}

With these results in hand, we are now ready to perform a sharp analysis of $\|\hat \theta - \theta^*\|_2$.

\begin{proposition}\label{prop:asymptotics}
There exist absolute constants $C, C' > 0$ such that the following holds for any $\delta \in (0,0.1)$ and any integer $L \ge 2$.
Fix $\theta^{(0)} \in \cH$ and let $\Delta^{(0)} := \max\{ 4L \|\theta^{(0)} - \theta^*\|_\infty, 6 \sqrt{L} \, \sigma \}$. 
Also, let $\Delta := 2 C \sigma \sqrt{\log \frac{N}{\delta}}$. 
Suppose that
\begin{subequations}
\label{eq:prop-2-cond-set}
\begin{align}
&N \ge C \max \Big\{ L^2 d^{\frac{L}{L-1}} \log \frac{dL}{\delta} , L^3 d \log \frac{L}{\delta} \Big\} , \qquad
N \gg L^4 \, d  \left(\log \frac{N}{\delta}\right)^5 ,
\label{eq:condition-N} \\
&\max_{i \in [d]} |S_i(\Delta^{(0)})| \le \frac{d}{C L^2} , \qquad
\max_{i \in [d]} |S_i(\Delta)| = o\left( \frac{d}{(L \log(N/\delta))^{3/2}} \right) ,
\label{eq:condition-Delta}
\end{align}
\end{subequations}
where $S_i(\cdot)$ is defined in \eqref{def:S_i}. 
Let $\hat \theta$ be the limit of the EM sequence $\{\theta^{(t)}\}_{t \ge 0}$ defined in \eqref{eq:iteration-em} as guaranteed by Proposition~\ref{prop:em-contraction}. 
Then we have 
$$
\|\hat \theta - \theta^*\|_2^2 \le (1+o(1)) \, \sigma^2 \frac{d-1}{2N} \tr(\hat \Sigma^\dag) + C' \sigma^2 \left( \frac{d^{3/2}}{N}\sqrt{\log\frac{1}{\delta}}+ \frac{d}{N} \log \frac{1}{\delta} \right) .
$$
\end{proposition}

\begin{proof}
There is an event $\cE$ of probability at least $1-\delta$ on which 
\begin{itemize} 
\item
$|y_r - x_r^\top \theta^*| \le C \sigma \sqrt{\log \frac{N}{\delta}}$ for all $r \in [N]$; 

\item
the bounds in Lemma~\ref{lem:high-prob-bds} hold;

\item
$\|\hat\Sigma\|_{op} \leq 3$, $\|\hat\Sigma^\dag\|_{op} \leq 5$, and $\tr(\hat\Sigma^\dag) \ge (d-1)/3$ by Proposition~\ref{prop:op_norm} and Lemma~\ref{lem:trace-inv-lower};

\item
$\|\hat\Sigma^\dag\|_\infty \le 2L$ by Proposition~\ref{prop:sigma-dag-bound};

\item
$\|\hat Q(\theta^*) - \theta^*\|_2^2 \le \sigma^2 \frac{d-1}{2N} \tr(\hat \Sigma^\dag) + C \sigma^2 \left( \frac{d^{3/2}}{N}\sqrt{\log\frac{1}{\delta}}+ \frac{d}{N} \log \frac{1}{\delta} \right)$
by Lemmas~\ref{lem:l2-err-exp} and~\ref{lem:theta_ell2_bd};

\item
$\|\hat \theta - \theta^*\|_\infty \le C \sigma \sqrt{\frac{L\,d}{N}\log\frac d\delta}$ by Proposition~\ref{prop:em-contraction};

\item
$\|A\|_\infty \le C' \left( \max_{i \in [d]} \frac{|S_i(\Delta)|}{d} + \frac{d}{N}  \log\frac d\delta \right) \log \frac{N}{\delta} \le o\left( \frac{1}{L^{3/2} \sqrt{\log(d/\delta)}} \right)$ by Lemma~\ref{lem:A-infty-norm-bd} and the assumptions \eqref{eq:condition-N} and \eqref{eq:condition-Delta};

\item
$\|\omega(\hat\theta,t)\|_\infty \le \frac{C'}{\sigma} \|\hat \theta - \theta^*\|_\infty^2 \left(\log \frac{N}{\delta}\right)^{3/2} \le C' C^2 \sigma \frac{L \, d}{N} \left(\log \frac{N}{\delta}\right)^{5/2} \le o\left( \frac{\sigma}{L} \sqrt{\frac{ d }{ N }} \, \right)$ by Lemma~\ref{lem:taylor_error}, the above bound on $\|\hat \theta - \theta^*\|_\infty$, and \eqref{eq:condition-N}. 
\end{itemize}
Let us condition on the event $\cE$ in the rest of the proof.

Since $\hat \Sigma \hat \theta = \hat \Sigma \hat Q(\hat \theta) = \bar Q(\hat \theta)$, we have 
\begin{align*}
\hat\Sigma(\hat\theta - \theta^*) &= \bar Q(\hat{\theta}) - \bar Q(\theta^*) + \bar Q(\theta^*) - \hat\Sigma\theta^* \\
&= \frac{d-1}{2N}\sum_{r = 1}^N\left(\tanh\left(\frac{y_rx_r^\top \hat\theta}{\sigma^2}\right) - \tanh\left(\frac{y_rx_r^\top \theta^*}{\sigma^2}\right)\right)y_rx_r + \bar Q(\theta^*) - \hat\Sigma\theta^*.
\end{align*}
By Taylor's theorem, it holds that 
$$
\tanh\left(\frac{y_rx_r^\top \hat\theta}{\sigma^2}\right) - \tanh\left(\frac{y_rx_r^\top \theta^*}{\sigma^2}\right)
= \tanh'\left(\frac{y_r x_r^\top \theta^*}{\sigma^2}\right)\frac{y_rx_r^\top(\hat{\theta} - \theta^*)}{\sigma^2} + \omega(\hat\theta, r, t_r),
$$
where the error term $\omega(\hat\theta, r, t_r)$ is defined in \eqref{eq:err-term} for some $t_r \in [0,1]$. 
With this in hand, we have
\begin{align*}
\hat \Sigma (\hat \theta - \theta^*) 
&= \frac{d-1}{2N} \sum_{r=1}^N \tanh'\left(\frac{y_r x_r^\top \theta^*}{\sigma^2}\right)\frac{y_rx_r^\top(\hat{\theta} - \theta^*)}{\sigma^2}  y_r x_r + \omega(\hat \theta, t) + \bar Q(\theta^*) - \hat\Sigma\theta^* \\
&= A (\hat \theta - \theta^*) + \bar Q(\theta^*) - \hat\Sigma\theta^* + \omega(\hat \theta, t) ,
\end{align*}
where $\omega(\hat \theta, t)$ is defined in \eqref{eq:err-term-sum} for $t = (t_1, \dots, t_d)$, and $A$ is defined in \eqref{eq:def-A}. 
It follows that
\begin{equation}
\hat \theta - \theta^*
= \hat Q(\theta^*) - \theta^* + \hat\Sigma^\dag A (\hat \theta - \theta^*) + \hat\Sigma^\dag \omega(\hat \theta, t) .
\label{eq:theta-hat-theta-star-expansion}
\end{equation}

On the event $\cE$, the main term $\hat Q(\theta^*) - \theta^*$ in the above formula satisfies
\begin{equation}
\|\hat Q(\theta^*) - \theta^*\|_2^2 \le \sigma^2 \frac{d-1}{2N} \tr(\hat \Sigma^\dag) + C \sigma^2 \left( \frac{d^{3/2}}{N}\sqrt{\log\frac{1}{\delta}}+ \frac{d}{N} \log \frac{1}{\delta} \right).
\label{eq:main-term-bd}
\end{equation}
The error term satisfies
\begin{align*}
\|\hat\Sigma^\dag A (\hat \theta - \theta^*) + \hat\Sigma^\dag \omega(\hat \theta, t) \|_\infty 
\le \|\hat\Sigma^\dag\|_\infty \|A\|_\infty \|\hat \theta - \theta^*\|_\infty + \|\hat\Sigma^\dag\|_\infty \|\omega(\hat \theta, t) \|_\infty 
\le o\left( \sigma \sqrt{d/N} \, \right) 
\end{align*}
in view of the bounds $\|\hat \Sigma^\dag\|_\infty \le 2L$, $\|\hat \theta - \theta^*\|_\infty \le C \sigma \sqrt{\frac{L\,d}{N}\log\frac d\delta}$, 
$\|A\|_\infty \le o\left( \frac{1}{L^{3/2} \sqrt{\log(d/\delta)}} \right)$, and
$\|\omega(\hat\theta,t)\|_\infty \le o\left( \frac{\sigma}{L} \sqrt{\frac{ d }{ N }} \, \right)$ on the event $\cE$.
It follows that
\begin{equation}
\|\hat\Sigma^\dag A (\hat \theta - \theta^*) + \hat\Sigma^\dag \omega(\hat \theta, t) \|_2^2
\le o\left( \sigma^2 d^2/N \right) .
\label{eq:err-term-bd}
\end{equation}
Finally, since $\tr(\hat \Sigma^\dag) \ge (d-1)/3$ on the event $\cE$, the first term of the bound in \eqref{eq:main-term-bd} dominates the bound in \eqref{eq:err-term-bd}.
Combining \eqref{eq:main-term-bd} and \eqref{eq:err-term-bd} with \eqref{eq:theta-hat-theta-star-expansion} then completes the proof.
\end{proof}

\subsection{Proof of main results}
\label{sec:pf-main-results}

\begin{proof}[Proof of Theorem~\ref{thm:em-contraction}]
The theorem follows from Proposition~\ref{prop:em-contraction} together with the assumption $\theta^* \in \Theta(\beta)$. 
We first check that the conditions in Theorem~\ref{thm:em-contraction} imply the conditions in \eqref{eq:cond-samp-init}.

\noindent {\bf Checking \eqref{eq:cond-1-sample-size}:}
On the one hand, and as claimed in the statement of the theorem, 
\begin{align} \label{eq:cond-thm}
N \ge  \max\{ d^{1+\rho}, N_0\}.
\end{align}
On the other hand, choosing $L = 1 + \lceil 2/\rho \rceil$ and $\delta = N^{-D}$ in Proposition~\ref{prop:em-contraction}, we obtain the sample size requirement
\begin{align} \label{eq:cond-prop}
N \ge C \max \Big\{ L^2 d^{\frac{L}{L-1}} \log \frac{dL}{\delta} , L^3 d \log \frac{L}{\delta} \Big\}.
\end{align}
It can be easily verified that for a sufficiently large constant $N_0 > 0$, \eqref{eq:cond-thm} implies~\eqref{eq:cond-prop}.

\noindent {\bf Checking \eqref{eq:cond-2-s-i-delta}:} Furthermore, since $\theta^* \in \Theta(\beta)$, if $|\theta_i^* - \theta_j^*| \leq \Delta$, then $|i-j| \le \Delta d/\beta$.
Therefore, we have 
$$
S_i(\Delta) 
= \{j \in [d] \setminus \{i\} : |\theta_i^* - \theta_j^*| \leq \Delta\}
\subset \{j \in [d] \setminus \{i\} : |i-j| \leq \Delta d/\beta\} ,
$$
so
$$
|S_i(\Delta)| \le 2 \Delta d/\beta .
$$
As a result, for $\Delta^{(0)} = \max\{ 4L \|\theta^{(0)} - \theta^*\|_\infty, 6 \sqrt{L} \, \sigma \}$, it holds that
\begin{equation}
\max_{i \in [d]} |S_i(\Delta^{(0)})| \le \max\{ 8L \|\theta^{(0)} - \theta^*\|_\infty d/\beta, 12 \sqrt{L} \, \sigma d/\beta \} \le \frac{d}{C L^2} 
\label{eq:s-i-delta-bd-inter}
\end{equation}
in view of the assumptions
$\|\theta^{(0)} - \theta^*\|_\infty \le c_1 \beta$ and $\sigma \le c_1 \beta$. 

It remains to compare $\tau$ in Theorem~\ref{thm:em-contraction} and Proposition~\ref{prop:em-contraction}: once we set $\delta = N^{-D}$, the two are the same up to multiplicative factors depending solely on the pair $(\rho, D)$.
\end{proof}

\begin{proof}[Proof of Theorem~\ref{thm:asymptotics}]
The theorem follows from Proposition~\ref{prop:asymptotics}.
Continuing from the proof of Theorem~\ref{thm:em-contraction}, we check the two additional conditions in \eqref{eq:prop-2-cond-set} compared to \eqref{eq:cond-samp-init}.

\noindent {\bf Checking \eqref{eq:condition-N}:}
Choosing $L = 1 + \lceil 2/\rho \rceil$ and $\delta = N^{-D}$ again, we can use the assumption $N \ge \max\{d^{1+\rho}, N_0\}$ to verify that
$$
N \gg L^4 \, d \left(\log \frac{N}{\delta}\right)^5 .
$$

\noindent {\bf Checking \eqref{eq:condition-Delta}:}
We again have $|S_i(\Delta)| \le 2 \Delta d/\beta$, so for $\Delta = 2 C \sigma \sqrt{\log \frac{N}{\delta}}$, it holds
$$
\max_{i \in [d]} |S_i(\Delta)| 
\le \frac{4 C \sigma d}{\beta} \sqrt{\log \frac{N}{\delta}} 
= o\left( \frac{d}{( L\log(N/\delta))^{3/2}} \right) 
$$
in view of the assumption
$\sigma = o\left( \frac{\beta}{(\log N)^{2}} \right) .$

\noindent {\bf Implication:}
Proposition~\ref{prop:asymptotics} then gives 
$$
\|\hat \theta - \theta^*\|_2^2 \le (1+o(1)) \, \sigma^2 \frac{d-1}{2N} \tr(\hat \Sigma^\dag) + C' \sigma^2 \left( \frac{d^{3/2}}{N}\sqrt{\log N}+ \frac{d}{N} \log N \right) .
$$
Finally, $\tr(\hat \Sigma^\dag) \ge (d-1)/3$ with high probability by Lemma~\ref{lem:trace-inv-lower}, so the first term above dominates the other terms when $d \gg \log N$.
Therefore, Theorem~\ref{thm:asymptotics} follows.
\end{proof}

\appendix

\section{Probability tools}

\begin{lemma}[Gaussian concentration, Theorem~5.2.2 of \cite{vershynin2018high}]
\label{lem:ver-5.2.2} 
Consider a random vector $X \sim \cN(0, I_n)$ and an $L$-Lipschitz function $f : \R^n \to \R$. Then we have
\[\Pr\{|f(X) - \EE[f(X)]| \geq t\} \leq 2\exp\left(-c \, t^2/L^2\right)\]
for an absolute constant $c>0$.
\end{lemma}

\begin{definition}[Convex concentration property]
\label{def:conv_conc_prop}
Let $X$ be a random vector in $\R^n$. We say that $X$ has the convex concentration property with constant $K$ if for every $1$-Lipschitz convex function $\phi\,:\, \R^n\to \R$, we have $\EE[|\phi(X)|] < \infty$, and for every $t > 0$,
\[\Pr\{|\phi(X) - \EE[\phi(X)]| \geq t\} \leq 2\exp\left(-t^2/K^2\right).\]
\end{definition}

\begin{lemma}[Theorem~2.3 of \cite{adamczak2015note}]
\label{lem:conv_conc}
Let $X$ be a mean zero random vector in $\R^n$. If $X$ has the convex concentration property with constant $K$, then for any $A \in \R^{n \times n}$ and any $t > 0$,
$$
\Pr\{|X^\top AX - \EE[X^\top AX]| \geq t\} 
\leq 2\exp\left(-\frac{1}{C}\min\left(\frac{t^2}{K^4\|A\|_F^2}, \frac{t}{K^2\|A\|_{op}}\right)\right) ,
$$
for an absolute constant $C>0$.
\end{lemma}

\begin{lemma}[Matrix Bernstein, Theorem~1.6.2 of \cite{tropp2015introduction}]\label{lem:matrix_bernstein}
Consider a finite sequence $\{X_k\}_{k=1}^N$ of independent, random, Hermitian matrices in $\R^{d \times d}$. Assume that
\[\EE[X_k] = 0 \quad \text{and} \quad \|X_k\|_{op} \leq L \quad \text{for each } k \in [N] .\]
Introduce the random matrix
\[Y = \sum_{k=1}^N X_k.\]
Let $v(Y)$ be the matrix variance statistic of the sum:
\[v(Y) = \|\EE[Y^2]\| = \bigg\|\sum_{k=1}^N \EE[X_k^2]\bigg\|.\]
Then
\[\EE[ \|Y\| ] \leq \sqrt{2v(Y) \log (2d)} + \frac{1}{3}\,L\,\log (2d).\]
Furthermore, for all $t > 0$,
\[\Pr\{ \|Y\| \geq t\} \leq 2 d\,\exp\left(-\frac{t^2/2}{v(Y) + Lt/3}\right).\]
\end{lemma}

\begin{lemma}[Lemma~2.27 of \cite{wainwright2019high}]\label{lem:wainwright}
Let $f \,:\, \R^N \to \R$ be a differentiable function. Then for every convex function $\phi\,:\,\R\to \R$, we have
\[\EE[\phi(f(X) - \EE[f(X)])] \leq \EE\left[\phi\left(\frac{\pi}{2}\langle\nabla f, Y\rangle\right)\right],\]
where $X,\,Y \sim \mathcal{N}(0, I_N)$ are independent standard Gaussians.
\end{lemma}

\section{Improved sample complexity}
\label{sec:d-polylog-d}
In Theorem~\ref{thm:em-contraction}, we assume the conditions
\begin{itemize}
\item
(sample size) $N \ge d^{1+\rho}$;

\item
(noise) $\sigma \le c_1 \beta$;

\item
(initialization) $\|\theta^{(0)} - \theta^*\|_\infty \le c_1 \beta$.
\end{itemize}
It is possible to improve the sample complexity at the cost of strengthening the other two conditions.
Namely, we may assume the alternative conditions
\begin{itemize}
\item
(sample size) $N \ge d \, (\log d)^{C_1}$;

\item
(noise) $\sigma \le \frac{c_2 \, \beta}{(\log d)^{5/2}}$;

\item
(initialization) $\|\theta^{(0)} - \theta^*\|_\infty \le \frac{c_2 \, \beta}{(\log d)^3}$.
\end{itemize}
To see the sufficiency of this set of conditions, we apply Proposition~\ref{prop:em-contraction} with $L = \big\lceil \frac{\log d}{\log \log d} \rceil$ and $\delta = N^{-D}$ for a constant $D > 0$.
It suffices to check the two conditions in \eqref{eq:cond-samp-init}.
We have
$$
d^{\frac{1}{L-1}} \le d^{\frac{2 \log \log d}{\log d}} = (\log d)^{2} .
$$
Therefore, if $N \ge d \, (\log d)^{C_1}$ for a sufficiently large constant $C_1 = C_1(D) > 0$, then
$$
N \ge C \max \Big\{ L^2 d^{\frac{L}{L-1}} \log \frac{dL}{\delta} , L^3 d \log \frac{L}{\delta} \Big\} .
$$
Further, as in \eqref{eq:s-i-delta-bd-inter}, we have
\begin{equation*}
\max_{i \in [d]} |S_i(\Delta^{(0)})| \le \max\{ 8L \|\theta^{(0)} - \theta^*\|_\infty d/\beta, 12 \sqrt{L} \, \sigma d/\beta \} \le \frac{d}{C L^2} 
\end{equation*}
provided that $\|\theta^{(0)} - \theta^*\|_\infty \le \frac{\beta}{8 C L^3}$ and $\sigma \le \frac{\beta}{12 C L^{5/2}}$.
These are satisfied once we choose $c_2 > 0$ to be sufficiently small.

\section{Minimax lower bounds with a fixed design}
\label{sec:lin-reg-rate}

If the signs $z_r$ are given in the model \eqref{eq:model_symmetric} and the covariates $x_r$ are fixed, then the problem reduces to linear regression with a fixed design.
It is an elementary fact that the least squares estimator $\hat \theta^{\mathsf{LS}}$ achieves the risk
$$
\E \|\hat \theta^{\mathsf{LS}} - \theta^*\|_2^2 = \sigma^2 \tr\bigg( \Big( \sum_{r=1}^N x_r x_r^\top \Big)^\dag \bigg) . 
$$
It is well-known that this risk is minimax-optimal.
To show that $\sigma^2 \tr\Big( \big( \sum_{r=1}^N x_r x_r^\top \big)^\dag \Big)$ is a lower bound on the minimax risk, the standard proof is via the Bayes risk assuming that $\theta^*$ has the prior distribution $\cN \Big( 0, \tau^2 \big( \sum_{r=1}^N x_r x_r^\top \big)^\dag \Big)$ for a parameter $\tau$.
One can compute the Bayes-optimal estimator which is the posterior mean of $\theta^*$.
Then the Bayes risk achieved by the posterior mean evaluates to $\frac{\tau^2}{1+\tau^2} \sigma^2 \tr \Big( \big( \sum_{r=1}^N x_r x_r^\top \big)^\dag \Big)$.
Since the Bayes risk is a lower bound on the minimax risk, letting $\tau \to \infty$ proves the desired lower bound.

If the signs $z_r$ are unknown in the mixture model \eqref{eq:model_symmetric} and the covariates $x_r$ are fixed, then $\sigma^2 \tr\Big( \big( \sum_{r=1}^N x_r x_r^\top \big)^\dag \Big)$ is still a lower bound on the minimax risk for estimating $\theta^*$.
Therefore, conditional on a typical realization of $(x_r)_{r=1}^N$, the proof of Theorem~\ref{thm:asymptotics} in fact shows that the EM fixed point $\hat \theta$ achieves the minimax rate with the sharp constant (see Theorem~\ref{thm:asymptotics} and Eq.~\eqref{eq:sharp-rate}).
Since this work assumes a random design of the covariates, to keep our terminology consistent, we do not further formalize a minimax result in the fixed-design setting.

\section*{Acknowledgments}
AD was supported in part by NSF grant DMS-2053333.
CM was supported in part by NSF grants DMS-2053333 and DMS-2210734.
AP was supported in part by NSF grants CCF-2107455 and DMS-2210734, and by Research Awards from Adobe and Amazon.
CM thanks Dylan J.\ Altschuler and Anderson Ye Zhang for helpful discussions.

\bibliographystyle{abbrvnat}
\bibliography{references}




\end{document}